\documentclass[a4paper,  12pt]{amsart}

\usepackage[margin=2cm]{geometry}
\usepackage{amscd}
\usepackage{amsmath}
\usepackage{amssymb}
\usepackage{amsthm}
\usepackage{verbatim}
\usepackage[all]{xy}
\usepackage{color}
\usepackage[pagebackref]{hyperref}

\newtheorem{theorem}{Theorem}[section]
\newtheorem{lemma}[theorem]{Lemma}
\newtheorem{proposition}[theorem]{Proposition}

\theoremstyle{plain}
\newtheorem{result}{Theorem}

\theoremstyle{definition}

\theoremstyle{definition} 
\newtheorem*{definition*}{Definition}

\theoremstyle{theorem}

\theoremstyle{theorem} 
\newtheorem*{conjecture*}{Conjecture}

\theoremstyle{definition}

\theoremstyle{definition} 
\newtheorem*{assumption*}{Assumption}

\theoremstyle{remark}

\theoremstyle{remark} 
\newtheorem*{remark*}{Remark}

\newcommand{\NN}{\mathbb{N}}
\newcommand{\RR}{\mathbb{R}}
\newcommand{\CC}{\mathbb{C}}

\newcommand{\ZZ}{\mathbb{Z}}
\newcommand{\LLL}{\mathcal{L}}
\newcommand{\D}{\mathcal{D}}

\newcommand{\Z}{\mathcal{Z}}
\newcommand{\V}{\mathcal{V}}

\newcommand{\TT}{\mathbb{T}}
\newcommand{\di}{\mathrm{d}}

\newcommand{\bj}{\boldsymbol{j}}
\newcommand{\bk}{\boldsymbol{k}}
\newcommand{\bl}{\boldsymbol{l}}
\newcommand{\bn}{\boldsymbol{n}}

\newcommand{\bep}{\boldsymbol{\e}}
\newcommand{\bmu}{\boldsymbol{\mu}}
\newcommand{\bnu}{\boldsymbol{\nu}}

\newcommand{\bZ}{\boldsymbol{\mathcal{Z}}}

\newcommand{\Sl}{\mathfrak{sl}}
\newcommand{\SL}{\mathrm{SL}}
\newcommand{\SO}{\mathrm{SO}}
\newcommand{\PSL}{\mathrm{PSL}}
\newcommand{\Cinf}{C^{\infty}}
\newcommand{\Lii}{L^{2}}
\newcommand{\aaa}{\mathfrak{a}}
\newcommand{\HH}{\mathcal{H}}
\newcommand{\II}{\mathcal{I}}

\providecommand{\abs}[1]{\lvert#1\rvert}
\providecommand{\Abs}[1]{\left|#1\right|}
\providecommand{\norm}[1]{\lVert#1\rVert}

\def\a{{\alpha}}
\def\b{{\beta}}
\def\d{{\delta}}
\def\e{{\epsilon}}
\def\s{{\sigma}}

\def\t{{\tau}}

\def\l{{\lambda}}

\def\o{{\omega}}
\def\O{{\Omega}}
\def\G{{\Gamma}}

\title[Higher cohomology for Anosov actions]{Higher cohomology for Anosov actions on certain homogeneous spaces}
\author[F. A. Ram{\'i}rez]{Felipe A. Ram{\'i}rez}
\thanks{The author was partially supported by NSF RTG number DMS-0602191}
\date{}
\address{School of Mathematics, University of Bristol, Bristol, UK}
\email{f.a.ramirez@bristol.ac.uk}

\begin{document}

%===============================================

\begin{abstract}
We study the smooth untwisted cohomology with real coefficients for the action on $[\SL(2,\RR) \times\dots\times \SL(2,\RR)]/\G$ by the subgroup of diagonal matrices, where $\G$ is an irreducible lattice.  In the top degree, we show that the obstructions to solving the coboundary equation come from distributions that are invariant under the action.  In intermediate degrees, we show that the cohomology trivializes.  It has been conjectured by A. and S. Katok that, for a standard partially hyperbolic $\RR^d$- or $\ZZ^d$-action, the obstructions to solving the top-degree coboundary equation are given by periodic orbits, in analogy to Liv{\v{s}}ic's theorem for Anosov flows, and that the intermediate cohomology trivializes, as it is known to do in the first degree, by work of Katok and Spatzier.  Katok and Katok proved their conjecture for abelian groups of toral automorphisms.  For diagonal subgroup actions on $\SL(2,\RR)^d/\G$, our results verify the ``intermediate cohomology'' part of the conjecture, and are a step in the direction of the ``top-degree cohomology'' part.
\end{abstract}

%===============================================

\maketitle

\tableofcontents

\section{Introduction}

The focus of this work is a conjecture (stated below, and in \cite{KK95}) due to A. Katok and S. Katok on the smooth cohomology for hyperbolic actions by higher-rank abelian groups on smooth manifolds.  The conjecture seeks to explain the contrast between rank-one and higher-rank $1$-cocycle rigidity phenomena for these group actions.  For Anosov flows and diffeomorphisms, Liv{\v{s}}ic's theorem (and subsequent extensions by several authors) gives a full set of obstructions to solving the degree-one coboundary equation.  These obstructions correspond to periodic orbits.  For higher-rank Anosov actions satisfying certain irreducibility conditions, results of A. Katok and R. Spatzier show that there are no obstructions to solving the degree-one (almost) coboundary equation.\footnote{In fact, they prove this for the so-called \emph{standard} Anosov actions, defined in \cite{KS94first}.  These constitute the known examples of higher rank Anosov actions satisfying the desired irreducibility conditions.}  Katok and Katok conjecture that any standard partially hyperbolic action by $\ZZ^d$ or $\RR^d$ has obstructions given by periodic data to solving the coboundary equation in degree $d$, but not in lower degrees.  This conjecture frames the theorems for first cohomology of Liv{\v{s}}ic and Katok--Spatzier in a general statement involving higher cohomology.

The full problem remains open.  At present, only the results of Katok and Katok exist in the literature \cite{KK95, KK05}.  There, they proved their conjecture for actions $\ZZ^d \curvearrowright \TT^N$ by partially hyperbolic toral automorphisms.  We treat the case of Anosov $\RR^d$-actions on quotients of $d$-fold products of $\SL(2,\RR)$.

%================================================

\subsection{Historical context and problem statement}

The first cohomology of group actions has been much studied in dynamics.  One of the most celebrated results in this area is from work of A. N. Liv{\v{s}}ic \cite{Livsic}, and subsequent related work by V. Guillemin and D. Kazhdan \cite{GK80}, and R. de la Llave, J. Marko, and R. Moriy{\'o}n \cite{dlLMM}, where it was established that for Anosov flows and diffeomorphisms, the first cohomology is determined by periodic orbits.  Liv{\v{s}}ic's theorem states that for an Anosov flow $\RR \curvearrowright M$, where $M$ is a smooth manifold, a given $1$-cocycle is a coboundary if and only if its integral around every closed orbit is $0$.  It is not hard to see that this condition from periodic orbits is necessary: one can see a $1$-cocycle over the flow as a closed differential $1$-form $\o$ along the orbit foliation.  By restricting $\o$ to any closed orbit, one obtains a $1$-form in the usual (de Rham) sense on this orbit.  If the form is exact, then the fundamental theorem of calculus implies that its integral over the orbit must be $0$.  Liv{\v{s}}ic's theorem gives that, in the hyperbolic setting, this condition coming from periodic orbits is also sufficient.  

For standard Anosov and partially hyperbolic $\RR^d$-actions, $d \geq 2$, the first cohomology was studied by A. Katok and R. Spatzier \cite{KS94first}.  Here, the situation is different.  Katok and Spatzier showed that the first smooth cohomology \emph{trivializes}, \emph{i.e.}~any smooth $\RR$-valued cocycle is cohomologous to a constant cocycle; it is an \emph{almost coboundary}.  In other words, for higher-rank Anosov actions, there are no obstructions to solving the almost coboundary equation.  This comes about as a consequence of having positive codegree.  One of the key steps in the proof of the Katok--Spatzier result is the so-called \emph{higher-rank trick}; it is essentially the observation that if one has a converging sum over two parameters, and one knows that every sum over the first parameter is the same, then one can conclude that every sum over the first parameter must be $0$.  This observation becomes useful in a natural way when solving the coboundary equation in positive codegree.  For example, in Katok and Spatzier's work, a distributional solution to the degree-one coboundary equation is obtained by summing over values of a given cocycle at evenly spaced points on a ``forward'' orbit along a fixed direction in the acting $\RR^d$.  A similar sum over the corresponding ``backward'' orbit yields another distributional solution which must be shown to coincide with the first one.  For this, their difference, which is now essentially a sum over points on a line in $\RR^d$, must be zero.  It is shown that a further sum over an independent direction in $\RR^d$ converges, and that the first sum is independent of this second parameter.  The observation described above then allows the conclusion that the two distributional solutions are the same.

Versions of the higher-rank trick have since appeared in other cocycle rigidity results, not just in the hyperbolic setting \cite{M2,M1,Ram09}, and a version of it also appears in the intermediate-degree part of the Katok--Katok result for hyperbolic toral automorphisms \cite{KK95}.  Indeed, our Proposition \ref{prop} can be thought of as a version of the higher-rank trick.

Adopting the point of view that $1$-cocycles are closed differential $1$-forms along the orbits, it is natural to define higher-degree cocycles by closed differential forms in the corresponding degree (see Section \ref{definitions}).  Now, consider a $d$-cocycle $\o$ over an Anosov $\RR^d$-action on $M$.  As in the case of flows, one can restrict $\o$ to a periodic orbit, and obtain a top-degree differential form over this orbit in the usual sense.  One sees from Stokes' theorem that a necessary condition for being able to solve the coboundary equation $\di\eta=\o$ is that the integral of $\o$ over every closed orbit is $0$.  We will say that the action satisfies the \emph{Liv{\v{s}}ic property for $d$-cocycles} if this condition is also sufficient for the existence of a solution to the coboundary equation.  Thus, we know from \cite{Livsic, GK80, dlLMM} that Anosov flows have the Liv{\v{s}}ic property for $1$-cocycles.

The following conjecture is due to A. and S. Katok \cite{KK95}, and is one of the principal motivations for our work.

\begin{conjecture*}[Katok--Katok]
Let $\a$ be a standard partially hyperbolic action of $\ZZ_{+}^{d}$, $\ZZ^{d}$, or $\RR^{d}$, $d\geq2$.  Then the smooth $n$-cohomology of $\a$ trivializes for $1\leq n\leq d-1$, and $\a$ satisfies the Liv{\v{s}}ic property for $d$-cocycles.  If $\a$ is a standard Anosov action the same is true in $C^1$ and H{\"o}lder cases.
\end{conjecture*}

The conjecture implies that the contrast between the rank-one and higher-rank situations (for $1$-cocycles) lies in the fact that for flows, the first cohomology is the top-degree cohomology, whereas for higher-rank actions it is not.  With this conjecture in mind, one expects that cohomology classes in $H^{d}(M)$ for an Anosov or partially hyperbolic $\RR^d$-action on $M$ are determined by integrals over closed orbits; for $1\leq n \leq d-1$, one expects $H^{n}(M) \cong \RR^{\binom{d}{n}}$---the cohomology classes are determined by constant functions on $M$.

Katok and Katok proved the conjecture for $\ZZ^d$-actions by partially hyperbolic toral automorphisms \cite{KK95, KK05}.  Their strategy was to pass to a dual problem on Fourier coefficients of functions on tori.  There, the natural obstructions to the coboundary equation are distributions on the torus that are invariant under the action.  (These are referred to in the paper as \emph{invariant pseudomeasures}.)  They proved that these are a complete set of obstructions, and that these obstructions are approximated by linear combinations of invariant measures supported on periodic orbits.  The latter is an extension of a corresponding result of W. Veech for a single partially hyperbolic toral automorphism \cite{Vee86}.

%================

\subsection{Statements of results} \label{results}

We consider the subgroup $A\cong\RR^d$ of diagonal matrices in the $d$-fold product 
\[
	\SL(2,\RR)^d := \SL(2,\RR)\times\dots\times\SL(2,\RR),
\] 
and its action $\RR^d \curvearrowright \SL(2,\RR)^{d}/\G$ on a quotient by an irreducible lattice $\G$.  Setting the following elements of the Lie algebra $\mathrm{Lie}(A) := \aaa \subset \Sl(2,\RR)^d$
\begin{align*}
	X_1 &= \left(\begin{pmatrix} 1/2 & 0 \\ 0 & -1/2 \end{pmatrix} , (0), \dots, (0) \right) \\
	X_2 &= \left((0), \begin{pmatrix} 1/2 & 0 \\ 0 & -1/2 \end{pmatrix} , (0), \dots, (0) \right) \\
		&\vdots \\
	X_d &= \left((0),\dots, (0), \begin{pmatrix} 1/2 & 0 \\ 0 & -1/2 \end{pmatrix}\right),
\end{align*} 
the coboundary equation for $d$-forms is equivalent to the problem of finding smooth functions 
\[
	g_1,\dots,g_d \in \Cinf(\SL(2,\RR)^{d}/\G)
\]
satisfying
\begin{align*}
	f &= X_1\,g_1 + \dots + X_d\,g_d
\end{align*}
for a given smooth function $f \in \Cinf(\SL(2,\RR)^{d}/\G)$.

We prove the following theorem, giving a complete set of obstructions to the top-degree coboundary equation.

\begin{theorem} \label{one}
	Let $\G \subset \SL(2,\RR)^d$ be an irreducible lattice.  If $f \in \Cinf(\Lii(\SL(2,\RR)^{d}/\G))$ is in the kernel of every $X_1,\dots,X_d$-invariant distribution, then there exist smooth functions 
	\[
		g_1,\dots,g_d \in \Cinf(\Lii(\SL(2,\RR)^{d}/\G))
	\]
	satisfying
	\begin{align*}
		f &= X_1\,g_1 + \dots + X_d\,g_d.
	\end{align*}
\end{theorem}

\begin{remark*}
	If it is true that the set of linear combinations of invariant measures supported on periodic orbits of $\RR^d \curvearrowright \SL(2,\RR)^{d}/\G$ is dense (in the weak topology) in the space $\II_{X_1,\dots,X_d}$ of invariant distributions, then Theorem \ref{one} implies that the action has the Liv{\v{s}}ic property for $d$-cocycles, as conjectured.  This would then have an application to Hilbert cusp forms, through a program, outlined by T. Foth and S. Katok, for finding spanning sets by relative Poincar{\'e} series associated to closed orbits.  This program was carried out by S. Katok for modular forms \cite{Kat85}, and later by Foth and Katok in other rank-one situations \cite{FK01}.   
\end{remark*}

The presence in Theorem \ref{one} of $\Cinf(\Lii(\SL(2,\RR)^{d}/\G))$ instead of $\Cinf(\SL(2,\RR)^{d}/\G)$ is due to the fact that, as with Katok and Katok's passage to a dual problem, we work primarily in the unitary dual of $\SL(2,\RR)^d$.  We take $\Cinf(\Lii(\SL(2,\RR)^{d}/\G))$ to be the set of smooth vectors, in the representation theoretic sense, of the left-regular representation of $\SL(2,\RR)^d$ on $\Lii(\SL(2,\RR)^{d}/\G)$.  For cocompact $\G$, the set of smooth vectors coincides with the set of smooth functions $\Cinf(\SL(2,\RR)^{d}/\G)$.

A representation theoretic version of Theorem \ref{one} (from which Theorem \ref{one} follows) appears as Theorem \ref{a} in Section \ref{mainresults}.  It is one of the main results of this paper, and is inspired by work of D. Mieczkowski \cite{M2}, where the $d=1$ case is proved, and work of L. Flaminio and G. Forni \cite{FF}, where the coboundary equation for horocycle flows is studied.  Our method relies on an inductive procedure for establishing Theorem \ref{a} for $(d+1)$-fold products, assuming it holds for $d$-fold products.  Thus, Mieczkowski's work provides our base case.

We have the following theorem for intermediate cohomology.

\begin{theorem} \label{two}
	Let $\G \subset \SL(2,\RR)^{d}$ be an irreducible lattice, $A \subset \SL(2,\RR)^{d}$ the subgroup of diagonal matrices.  Then the smooth $n$-cohomology of the $A$-action on $\SL(2,\RR)^{d}/\G$ trivializes for $1 \leq n \leq d-1$.  
\end{theorem}

\begin{remark*}
Theorem \ref{two} verifies the ``intermediate cohomology'' part of Katok and Katok's conjecture for the actions we consider.
\end{remark*}

Again, we have a representation theoretic version of Theorem \ref{two}, listed as Theorem \ref{b} in Section \ref{mainresults}.  It is our second main result.  Again, we employ an inductive procedure to establish the theorem for $d$-fold products, assuming it is true for $(d-1)$-fold products.  This induction is analogous to one used by Katok and Katok to establish the ``intermediate cohomology'' part of their conjecture for toral automorphisms \cite{KK95} .  Our base case can be taken to come from Katok and Spatzier's work applied to Anosov $\RR^2$-actions \cite{KS94first}, or Mieczkowski's results in \cite{M1}.  

%===========================================================

\subsection{A note on the prospect of treating other cases}

Standard Anosov actions either come from: $1$) automorphisms of tori and nilmanifolds; $2$) actions on homogeneous spaces $G/\G$ by a split Cartan subgroup of $G$; or, $3$) a version of  the latter that is ``twisted'' by the former (through a construction found in \cite{KS94first}).  Our results fit into the second category, where the natural tools for addressing problems tend to be more complicated than in the first category.  For instance, the problem of computing higher-degree cohomology becomes significantly more involved, at least from a representation-theoretic point of view, once $\SL(2,\RR)$ is replaced with other semisimple Lie groups.  Even for other \emph{real-rank-one} simple Lie groups the representation theory is much more complicated, although there is some hope of adapting our induction scheme to this setting.  Such an adaptation requires not only a detailed picture of the unitary duals of these groups and their spaces of invariant distributions (akin to the pictures we have here for $\SL(2,\RR)$), but also an appreciable reinterpretation of the induction methods developed in this article.  This is the focus of a work in progress.

%================================================

\subsection{Definitions} \label{definitions}

Let $\RR^d \curvearrowright M$ be a locally free action by diffeomorphisms on a smooth manifold $M$.  For $1\leq n\leq d$, we define an \emph{$n$-form over the action} to be a smooth assignment taking each point $x \in M$ to a map
\[
	\o_{x}:(T_{x} (\RR^{d}\cdot x))^n \cong (\RR^{d})^{n}\rightarrow \RR
\]
that is multi-linear and skew-symmetric, where $\RR^{d}\cdot x$ denotes the orbit of $x$, and its tangent space is denoted $T_{x}(\RR^{d}\cdot x) \subset T_{x}M$; it is naturally identified with $\RR^d$.  We use $\O_{\RR^d}^{n}(M)$ to denote the set of $n$-forms over the action $\RR^d \curvearrowright M$, often dropping the subscript when there is no risk of confusion.  

The \emph{exterior derivative} for $\RR^d \curvearrowright M$ maps $n$-forms to $(n+1)$-forms by
\[
	\di\o_{x}(V_1,\dots, V_{n+1}) := \sum_{j=1}^{n}(-1)^{j+1}\,V_{j}\,\o_{x}(V_1,\dots,\widehat{V_j},\dots,V_{n+1}),
\]
with ``$\quad\widehat{}\quad$'' denoting omission.  One can check that $\di^2=0$.

An $n$-form $\o$ is said to be \emph{closed}, and is called a \emph{cocycle}, if $\di\o=0$.  It is said to be \emph{exact}, and is called a \emph{coboundary}, if there there is an $(n-1)$-form $\eta$ satisfying $\di\eta=\o$.  Two $n$-forms are said to be \emph{cohomologous} if they differ by a coboundary.  We are interested in the set of cohomology classes, $H^{n}(M)$.  

The first cohomology $H^{1}(M)$ coincides with the set of equivalence classes of smooth $\RR$-valued cocycles in the usual dynamical sense.  That is, a smooth $\RR$-valued cocycle over the action $\RR^d \curvearrowright M$ is usually defined in dynamics as a smooth function $\a:\RR^d \times M \rightarrow \RR$ satisfying the \emph{cocycle identity}:
\[
	\a(r_1 + r_2, x) = \a(r_1, r_{2}.x)+\a(r_2,x)
\]
for all $r_{1},r_{2} \in \RR^d$ and $x \in M$.  Two smooth cocycles $\a_1$ and $\a_2$ are said to be \emph{smoothly cohomologous}, according to the usual dynamical definitions, if there is a smooth function $P:M \rightarrow \RR$ satisfying
\begin{align*}
	\a_{1}(r,x) = -P(r.x) + \a_{2}(r,x) + P(x).
\end{align*}
By differentiating $\a_{1}$ and $\a_{2}$ in directions $V \in \RR^d$, we obtain $1$-forms $\o_{1}(V)$ and $\o_{2}(V)$ in the sense described above.  That $\a_1$ and $\a_2$ satisfy the cocycle identity implies that $\o_1$ and $\o_2$ are closed.  That $\a_1$ and $\a_2$ are smoothly cohomologous is equivalent to the existence of a smooth function $P \in \Cinf(M)$, or $0$-form, satisfying $\di P = \o_{2}-\o_{1}$.

%=============================================

\section{Main results} \label{mainresults}

We work in the Hilbert space $\HH$ of a unitary representation of 
\[
\SL(2,\RR)^d := \SL(2,\RR) \times\dots\times \SL(2,\RR).  
\]
One of our goals is to solve the degree-$d$ coboundary equation
\begin{align} \label{cobeq}
		X_{1}\,g_1 + X_{2}\,g_2 + \dots + X_d\, g_d &= f
\end{align}
for a given $f \in \HH$ (see Section \ref{results}).

We will find that obstructions to solving equation \eqref{cobeq} come from elements of $\mathcal{E}^{\prime}(\HH)$ (the dual space of $\Cinf(\HH)$) that are invariant under the vectors $X_1,\dots,X_d$.  Let 
\[
\II_{X_1, \dots, X_d}(\HH) = \{\mathcal{D} \in \mathcal{E}^{\prime}(\HH) \mid \LLL_{X_i}\D=0 \quad \textrm{for all} \quad i=1,\dots,d \},
\]
where $\LLL_{X_i}$ is the Lie derivative operator.  The condition that $\LLL_{X_i}\D = 0$ is equivalent to the condition that $\D(X_i\,h)=0$ for all $h \in \Cinf(\HH)$.  The set $\II_{X_1,\dots,X_d}$ is exactly the set of distributions that are invariant under all of $X_1, \dots, X_d$.  

In fact, most of our work takes place in \emph{Sobolev spaces} for the representation on $\HH$.  The Sobolev space $W^\t (\HH)$ of order $\t \in \RR^{+}$ is the maximal domain in $\HH$ of the operator $(I + \Delta)^{\t/2}$, where $I$ is the identity operator and $\Delta$ is the Laplacian operator from $\SL(2,\RR)^d$.  $W^\t (\HH)$ is itself a Hilbert space, with inner product $\langle f,g \rangle_{\t} := \langle(I+\Delta)^{\t}f,g \rangle_{\HH}$.  The dual space of $W^\t (\HH)$ is denoted by $W^{-\t}(\HH)$, and is a subspace of the space $\mathcal{E}^{\prime}(\HH)$ of distributions.  We will find that obstructions to solving the coboundary equation in Sobolev spaces come from elements of 
\[
\II_{X_1, \dots, X_d}^{\t}(\HH) = \{\mathcal{D} \in W^{-\t}(\HH) \mid \LLL_{X_i}\D=0 \quad \textrm{for all} \quad i=1,\dots,d \},
\]
the set of distributions of Sobolev order $\t \in \RR$ that are $X_1,\dots,X_d$-invariant.

It is obviously necessary for $f$ to be in the kernel of all such distributions.  We show that this is also sufficient.

\begin{result}[Top-degree cohomology] \label{a}
	Let $\HH$ be the Hilbert space of a unitary representation of $\SL(2,\RR)^d$.  If there exists $\mu_0 > 0$ such that $\s(\Box_{i}) \cap (0,\mu_0) = \emptyset$ for all $i=1,\dots, d$, then the following holds.  For any $s > 1$, and $t < s-1$, there is a constant $C_{\mu_0,s,t}$ such that, for every $f \in \ker \II_{X_1,\dots,X_d}^{s_d}(\HH)$, where $s_d = 2^{d-1}s + \sum_{i=1}^{d-1}2^{i-1}(2s+d-i)$, there exist $g_1,\dots, g_d \in W^{t}(\HH)$ satisfying the coboundary equation \eqref{cobeq} for $f$, and satisfying the Sobolev estimates
\begin{align*} 
	\left\| g_{i} \right\|_{t} &\leq C_{\mu_0,s,t}\, \left\| f \right\|_{s_d}
\end{align*}  
for $i=1, \dots, d$.
\end{result}

\begin{remark*}
	There are precise definitions for $\Box_i$ and $\mu_0$ in Section \ref{gap}.  For now, it is worth remarking that the condition in Theorem \ref{a} involving these is a ``spectral gap'' condition on the representation of $\SL(2,\RR)^d$ on $\HH$ for the Casimir operators $\Box_i$ corresponding to the $d$ copies of $\SL(2,\RR)$.  Later, we will do most of our work in irreducible unitary representations, and the process of building global solutions for (reducible) unitary representations will depend on this spectral gap assumption.
\end{remark*}

Theorem \ref{a} is a generalization of the following result of Mieczkowski \cite{M2}, which provides the base case $d=1$ for an induction argument in our proof.  

\begin{theorem}[Mieczkowski] \label{m}
	Let $\HH_\mu$ be the Hilbert space of an irreducible unitary representation of $\SL(2,\RR)$, and $s > 1$.  If $\mu > \mu_0 > 0$ then there exists a constant $C_{\mu_0,s,t}$ such that, for all $f \in W^{s}(\HH_\mu)$,
	\begin{itemize}
		\item if $t < -1$, or 
		\item if $t < s-1$ and $\D(f)=0$ for all $\D \in \II^{s}(\HH_\mu)$, 
	\end{itemize}
	then the equation $X\,g = f$ has a solution $g \in W^{t}(\HH_\mu)$, which satisfies the Sobolev estimate $\left\| g \right\|_{t} \leq C_{\mu_0,s,t} \left\| f \right\|_{s}$.
\end{theorem}

\begin{remark*}
One can state a version of Theorem \ref{m} for any unitary representation of $\SL(2,\RR)$ that has a spectral gap for the Casimir operator $\Box$, by picking $\mu_0$ to work in every non-trivial sub-representation, as we have done with our statement of Theorem \ref{a}.  We have chosen to state the irreducible version of Theorem \ref{m} because it will be applied directly as the base case of our induction.

We also remark that Theorem \ref{m} is only proved for $\PSL(2,\RR)$ in \cite{M2}.  But some calculations show that it is also valid for the unitary dual of $\SL(2,\RR)$.  These are in Appendix \ref{appendix}.
\end{remark*}

Theorems \ref{a} and \ref{m}, and their methods of proof, are inspired by the following analogous theorem of Flaminio and Forni, for horocycle flows \cite{FF}.  
	
\begin{theorem}[Flaminio--Forni] \label{ff}
		Let
		\[
			U:= \begin{pmatrix}0 & 1 \\ 0 & 0 \end{pmatrix} \in \Sl(2,\RR),
		\]
		and $\HH$ the Hilbert space of a unitary representation of $\PSL(2,\RR)$ such that there exists $\mu_0 >0$ with $\s(\Box) \cap (0,\mu_0) = \emptyset$.  Then the following holds.  Let $\nu_0 \in [0,1)$ be defined by
		\[
			\nu_0 := \begin{cases}\sqrt{1-4\mu_0} &\textrm{if } \mu_0 < \frac{1}{4}; \\ 0 &\textrm{if } \mu_0 \geq \frac{1}{4}. \end{cases}
		\]
		Let $s>\frac{1+\nu_0}{2}$ and $t \in \RR$.  There exists a constant $C:=C_{\nu_{0},s,t}$ such that, for all $f \in W^{s}(\HH)$,
		\begin{itemize}
  			\item if $t < -\frac{1+\nu_0}{2}$ and $f$ has no component on the trivial sub-representation of $\HH$, or
  			\item if $t < s-1$ and $D(f)=0$ for all $D \in \II_{U}^{s}(\HH_{\mu})$,
		\end{itemize}
		then the equation $U \, g = f$ has a solution $g \in W^{t}(\HH_{\mu})$, which satisfies the Sobolev estimate
		\[
			\left\|g \right\|_{t} \leq C_{\mu_{0},s,t} \left\|f \right\|_{s}.
		\]
		A solution $g \in W^{t}(\HH)$ of the equation $U\,g=f$ is unique modulo the trivial sub-representa\-tion if and only if $t \geq -\frac{1-\nu_0}{2}$.
\end{theorem}

\begin{remark*}
Theorem \ref{ff} was also only proved for representations of $\PSL(2,\RR)$.  However, it remains valid if $\PSL(2,\RR)$ is replaced with $\SL(2,\RR)$.  Again, this is just a matter of carrying out some calculations in the irreducible unitary representations of $\SL(2,\RR)$ that are not irreducible unitary representations of $\PSL(2,\RR)$.  
\end{remark*}

Before stating the next result, let us briefly define forms in the context of representations.  Let $\HH$ be the Hilbert space of a unitary representation of $\SL(2,\RR)^d$.  We take an $n$-form (of Sobolev order at least $\t$) over the $\RR^d$-action on $\HH$ to be a map 
\[
\o:(\mathrm{Lie}(\RR^{d}))^n \rightarrow W^{\t}(\HH)
\] 
which is linear and anti-symmetric.  The exterior derivative, cocycles, coboundary equation, and cohomology are then defined in the usual way.  Use $\O_{\RR^d}^{n}(W^{\t}(\HH))$ to denote the set of $n$-forms of Sobolev order $\t$ over the $\RR^d$-action on $\HH$.  We prove the following.

\begin{result}[Intermediate cohomology] \label{b}
Let $\HH$ be the Hilbert space of a unitary representation of $\SL(2,\RR)^d$ such that almost every irreducible representation appearing in its direct decomposition has no trivial factor.  If there exists $\mu_0 > 0$ such that $\s(\Box_{i}) \cap (0,\mu_0) = \emptyset$ for all $i=1,\dots, d$, then the following holds.  Let $1\leq n\leq d-1$.  For any $s > 1$ and $t < s-1$, there is a constant $C_{\mu_0,s,t}$ such that, for any $n$-cocycle $\o \in \O_{\RR^d}^{n} (W^{s_d}(\HH))$, there exists $\eta \in \O_{\RR^d}^{n-1} (W^{t}(\HH))$ with $\di\eta = \o$ and
\begin{align*}
	\left\|\eta(X_{i_1},\dots,X_{i_{n-1}})\right\|_{t} &\leq C_{\mu_0,s,t}\min\left\{\left\| \o(X_{j_1},\dots,X_{j_n}) \right\|_{s_d}\right\}
\end{align*}
for all multi-indices $1\leq i_1 < \dots < i_{n-1} \leq d$, where the minimum is taken over all multi-indices $1\leq j_1 < \dots < j_{n} \leq d$ that become $i_1, \dots, i_{n-1}$ after omission of one index.
\end{result}

\begin{remark*}
The number $s_d$ is defined the same way in Theorem~\ref{b} as it is in Theorem~\ref{a}.
\end{remark*}

Theorems \ref{one} and \ref{two} follow from Theorems \ref{a} and \ref{b} by setting
\[
	\HH = \Lii_{0}(\SL(2,\RR)^{d}/\G)
\]
and noting that this representation satisfies the spectral gap assumption, by combining work of D. Kleinbock and G. Margulis in \cite{KM}, and L. Clozel in \cite{Clo03}.  (The relevant theorem is listed in Section \ref{sectionsix} as Theorem \ref{spectralgapth}.)

%===============================================

\section{Unitary representations of $\SL(2,\RR)$}

The purpose of this section is to recall some essential aspects of the unitary representation theory of the group $\SL(2,\RR)$.  

We note that this exposition follows \cite{FF} and \cite{M2} very closely, with the appropriate changes made to suit the fact that we are now working with $\SL(2,\RR)$, rather than $\PSL(2,\RR)$.  For detailed exposition of the unitary dual of $\SL(2,\RR)$, one can consult \cite{HT92, Lan75, Taylor}.

Fix the following generators of the Lie algebra $\Sl(2,\RR)$:
\[
	X = \begin{pmatrix} 1/2 & 0 \\ 0 & -1/2\end{pmatrix};\quad
	Y = \begin{pmatrix} 0 & -1/2 \\ -1/2 & 0\end{pmatrix};\quad
	\Theta = \begin{pmatrix} 0 & 1/2 \\ -1/2 & 0\end{pmatrix}.
\]
The Laplacian operator is defined by $\Delta = -X^2 - Y^2 - \Theta^2$ and the Casimir operator is defined by $\Box = -X^2 - Y^2 + \Theta^2$.  The Casimir operator $\Box$ is in the center of the universal enveloping algebra of $\Sl(2,\RR)$, and so it acts as a scalar in any irreducible unitary representation of $\SL(2,\RR)$.  In fact, this scalar parametrizes the unitary dual of $\PSL(2,\RR)$.  We will denote by $\HH_\mu$ the Hilbert space of an irreducible unitary representation of $\SL(2,\RR)$ where $\Box$ acts by multiplication by 
\[
\mu \in \RR^{+} \cup \{-n^{2}+n \mid n \in \ZZ^{+}\}\cup\{-n^{2}+\frac{1}{4} \mid n \in \ZZ_{\geq 0}\}.  
\]
It is also useful to define $\nu$ as a complex solution to 
\[
\nu^2 = 1 - 4\mu, 
\]
so that $\nu \in i\RR \cup \{2n - 1 \mid n \in \ZZ^{+} \} \cup \{2n \mid n \in \ZZ_{\geq 0}\} \cup (-1,1)\backslash\{0\}$.

The irreducible unitary representation space $\HH_\mu$ decomposes as
\[
	\HH_\mu = \overline{\bigoplus_{k \in \ZZ} m_k\,\V_{k}},
\]
where $\V_k \cong \CC$ is the irreducible representation of $\SO(2) \subset \SL(2,\RR)$ where the operator $-2i\Theta$ acts by multiplication by $k \in \ZZ$, and $m_k$ is the multiplicity with which this representation appears.  We identify five different possibilities: 
\begin{enumerate}
	\item \emph{First principal series.}  $\mu \geq \frac{1}{4}$, $\nu \in i\RR$, and 
		\[
			m_k = \begin{cases} 1\quad\mathrm{if}\quad k \in 2\ZZ \\ 0 \quad\mathrm{otherwise}. \end{cases}
		\]
	\item \emph{Second principal series.} $\mu > \frac{1}{4}$, $\nu \in i\RR\backslash\{0\}$, and 
		\[
			m_k = \begin{cases} 1\quad\mathrm{if}\quad k \in 2\ZZ+1 \\ 0 \quad\mathrm{otherwise}. \end{cases}
		\]
	\item \emph{Complementary principal series.} $0< \mu < \frac{1}{4}$, $\nu \in (-1,1)\backslash\{0\}$, and 
		\[
			m_k = \begin{cases} 1\quad\mathrm{if}\quad k \in 2\ZZ \\ 0 \quad\mathrm{otherwise}. \end{cases}
		\]
	\item \emph{First holomorphic discrete series.} $\mu  = -n^2 + n$ for some $n \in \ZZ^{+}$, $\nu = 2n-1$, and 
		\[
			m_k = \begin{cases} 1\quad\mathrm{if}\quad k \in [2n,\infty)\cap2\ZZ \\ 0 \quad\mathrm{otherwise}. \end{cases}
		\]
	\item \emph{Second holomorphic discrete series.} $\mu  = -n^2 + \frac{1}{4}$ for some $n \in \ZZ^{+}$, $\nu = 2n$, and 
		\[
			m_k = \begin{cases} 1\quad\mathrm{if}\quad k \in [2n+1,\infty)\cap2\ZZ+1 \\ 0 \quad\mathrm{otherwise}. \end{cases}
		\]
\end{enumerate}
\begin{remark*}
	The anti-holomorphic discrete series are not listed here because they are equivalent to their holomorphic counterparts.  It is also worth mentioning that ``second holomorphic discrete series'' does not seem to be standard terminology.  Most sources just list a holomorphic discrete series (and equivalent anti-holomorphic discrete series) without distinguishing between even and odd spectra.  However, it is convenient in this work to make the distinction.  We also remark that for each series, there are standard models that realize the action of $\SL(2,\RR)$, but we do not need them in our work.  
\end{remark*}

%=================================================

\subsubsection{Some useful indices}

Define
\[
	i_{\nu} = \left\lfloor \frac{1 + \mathfrak{R}(\nu)}{2} \right\rfloor
\]
so that 
\[
	i_{\nu} = \begin{cases} 0 \quad\textrm{for principal series and complementary series} \\ n \quad\textrm{for holomorphic discrete series.} \end{cases}
\]
Define $\e:= \e(\HH_\mu)$ by
\[
	\e = \begin{cases} 0 \quad\textrm{for first principal, first holomorphic discrete, and complementary series} \\ 1 \quad\textrm{for second principal and second holomorphic discrete series.} \end{cases}
\]
This notation is used for convenience.  Below, $i_{\nu}$ is used as the ``starting point'' for an indexing set $J_\nu$ for the spectrum of $(-2i\Theta)$, and $\e$ determines the parity of this spectrum, \emph{i.e.} it determines whether it consists of even numbers or odd numbers.  It is also what distinguishes the unitary dual of $\SL(2,\RR)$ from that of $\PSL(2,\RR)$.  For the latter, $\e = 0$ always.  Therefore, if one sets $\e = 0$ in what follows, then one recovers the setup from \cite{FF} and \cite{M2}.

We make extensive use of the following orthogonal basis for each irreducible $\HH_{\mu}$.

%===============================================

\subsection{An orthogonal basis for irreducible unitary representations} \label{orthobasis}

This section follows \cite{FF} in defining an orthogonal basis $\{u_{k}\}$ for $\HH_{\mu}$.

The elements 
\[
\eta_+ = X - i\,Y \qquad\mathrm{and}\qquad \eta_- = X + i\,Y
\]
of the universal enveloping algebra of $\Sl(2,\RR)$ raise and lower eigenvalues of the operator $-2i\,\Theta$ by $2$.  

Fixing a unit vector $v_{i_{\nu}} \in \V_{2 i_{\nu} + \e}$, we obtain an orthogonal basis for $\HH_\mu$ by 
\[
	\dots, \eta_{-}^{2}\,v_{0}, \eta_{-}\,v_{0}, v_{0},\eta_{+}\,v_{0}, \eta_{+}^{2}\,v_{0}, \dots
\]
for the principal and complementary series, and
\[
	v_{n},\eta_{+}\,v_{n}, \eta_{+}^{2}\,v_{n}, \dots
\]
for the discrete series.    

We re-scale these by defining
\begin{align*}
	u_k = c_k\,(\eta_{+}\,u_{k-1}), \quad&c_{k} = \frac{2}{2k + \e -1+\nu} \quad \mathrm{for}\quad k>i_{\nu} \\
	u_k = c_k\,(\eta_{-}\,u_{k+1}), \quad&c_{k} = \frac{2}{-2k - \e -1+\nu} \quad \mathrm{for}\quad k<i_{\nu}
\end{align*}
and $u_{i_{\nu}} = v_{i_{\nu}}$.  Note that for the discrete series, there is no $k < i_{\nu}$, so the second line does not apply. Defining
\[
	J_{\nu} := \begin{cases} \ZZ \quad\mathrm{if}\quad \mu \in \RR^{+} \\ [n,\infty) \cap \ZZ \quad\mathrm{if}\quad \mu \in \{-n^2 + n\}\cup\{-n^2 + \frac{1}{4}\}, \end{cases}
\]
we have a basis $\{u_{k}\}_{k \in J_{\nu}}$, where $u_k \in \V_{2k+\e}$.  It is also convenient to define
\begin{align}
	\Z_{\mu} &:= \begin{cases} \ZZ \quad\mathrm{if}\quad \mu \in \RR^{+} \\ \NN\cup\{0\} \quad\mathrm{if}\quad \mu \in \{-n^2 + n\}\cup\{-n^2 + \frac{1}{4}\}, \end{cases} \nonumber \\
		&= J_{\nu} - i_{\nu} \label{defZ}
\end{align}
so that the basis can also be written $\{u_{i_\nu + k}\}_{k \in \Z_{\mu}}$, where $u_{i_\nu + k} \in \V_{2i_{\nu} + 2k+\e}$.

One can compute the norms for the $\{u_k\}$ by the following calculations.  First, for $k>0$ (or $k>n$ for the discrete series),
\begin{align*}
	\left\| u_k \right\|^{2} &= \left\| c_k \right\|^2\,\langle \eta_{+}u_{k-1}, \eta_{+}u_{k-1} \rangle \\
		&= -\left\| c_k \right\|^2\,\langle \eta_{-}\eta_{+}u_{k-1}, u_{k-1} \rangle,
\end{align*}
since $\eta_{+}^{*} = -\eta_{-}$.  Then, observing that $\eta_{-}\eta_{+} = \Theta^2 + i\Theta - \Box$,
\begin{align*}
	&=  -\left\| c_k \right\|^2\,\langle (\Theta^2 + i\Theta - \Box)u_{k-1}, u_{k-1} \rangle \\
	&= \frac{2k + \e -1 -\nu}{2k + \e -1 + \bar\nu}\left\| u_{k-1} \right\|^{2}.
\end{align*}
For $k \leq 0$, one gets
\begin{align*}
	\left\| u_{k} \right\|^{2} &= \frac{|2k + \e| -1 - \nu}{|2k + \e| -1 + \bar\nu}\left\| u_{k+1} \right\|^{2},
\end{align*}
allowing us to conclude
\[
	\left\| u_{k} \right\|^{2} = \begin{cases} \left\| u_{k-1}\right\|^2 \quad\mathrm{if}\quad \nu \in i\RR \\
															\frac{|2k + \e| - 1 - \nu}{|2k + \e| - 1 + \nu}\left\| u_{k-1}\right\|^2 \quad\mathrm{if}\quad \nu \in \RR \end{cases}
\]
for $k > 0$ and 
\[
	\left\| u_{k} \right\|^{2} = \begin{cases} \left\| u_{k+1}\right\|^2 \quad\mathrm{if}\quad \nu \in i\RR \\
															\frac{|2k + \e| - 1 - \nu}{|2k + \e| - 1 + \nu}\left\| u_{k+1}\right\|^2 \quad\mathrm{if}\quad \nu \in \RR \end{cases}
\]
for $k <0$.  

Introducing
\begin{equation}\label{Pidef}
	\Pi_{\nu, \e,k} = \prod_{i = i_{\nu}+1}^{k}\frac{2i+\e-1-\nu}{2i+\e-1+\bar\nu} \quad\textrm{for any}\quad k \geq i_{\nu},
\end{equation}
where empty products are set to $1$, we have
\[
	\left\| u_k\right\|^2 = \left|\Pi_{\nu, \e, |k|}\right|.
\]

The following lemma is from \cite{FF}.

\begin{lemma} [Flaminio--Forni] \label{fflemma}
	If $\nu \in i\RR$ (principal series), for all $k \geq i_{\nu}=0$,
	\begin{align*}
		|\Pi_{\nu,\e,k}| = 1.
	\end{align*}
	There exists $C>0$ such that, if $\nu \in (-1,1) - \{0\}$ (complementary series), for all $k > i_{\nu} = 0$, we have
	\begin{align*}
		C^{-1}\left(\frac{1-\nu}{1+\nu}\right)(1+k)^{-\nu} \leq \Pi_{\nu,\e,k} \leq C\left(\frac{1-\nu}{1+\nu}\right)(1+k)^{-\nu};
	\end{align*}
	if $\nu = 2n+\e-1$, for all $k\geq l \geq i_{\nu}=n$ (discrete series), we have
	\begin{align*}
		C^{-1}\left(\frac{k-n -\e +1}{l-n-\e+1}\right)^{-\nu} \leq \frac{\Pi_{\nu,\e,k}}{\Pi_{\nu,\e,l}} \leq C\left(\frac{k-n+\nu}{l-n+\nu}\right)^{-\nu}.
	\end{align*}
\end{lemma}

\begin{remark*}
	Since Flaminio and Forni only work with representations of $\PSL(2,\RR)$, they prove the above lemma for $\e = 0$.  It is immediate that for the second principal series, $|\Pi_{\nu,\e,k}| = 1$, so it is only left to check that Lemma \ref{fflemma} holds in the second holomorphic discrete series, where $\nu = 2n$ and $\e=1$.  This is easily seen, by carrying out the proof in \cite{FF} and making the appropriate adjustments.  The result is listed in \cite{FF} as Lemma~$2.1$ and is proved in their Appendix.
\end{remark*}

%===============================================

\subsubsection{Sobolev norms}

Sobolev norms of the basis elements $u_k$ are computed by
\begin{align*}
	\left\| u_k \right\|_{\t}^{2} &= \langle (1 + \Delta)^{\t}u_k, u_k \rangle \\
		&= \langle (1 + \Box  - 2\Theta^2)^{\t}u_k, u_k \rangle \\
		&= (1 + \mu + 2(k + \e/2)^2)^{\t} \left\| u_k \right\|^2,
\end{align*}
and so a vector $\sum f(k)\,u_k \in \HH_\mu$ is in the Sobolev space $W^{\t}(\HH_\mu)$ if and only if the sum
\[
	\left\| f \right\|_{\t} = \left( \sum_{-\infty}^{\infty}(1 + \mu + 2(k + \e/2)^2)^{\t} \left|\Pi_{\nu,\e,|k|}\right| |f(k)|^2 \right)^{\frac{1}{2}}
\]
converges.

%===============================================

\subsubsection{Action of $X$ on the basis} \label{actionsection}

One can compute the action of $X$ on an element $u_k$ of our basis.  We have the following lemma, which is also an $\SL(2,\RR)$-version of one found in \cite{FF}:

\begin{lemma}[Flaminio--Forni] \label{action}
Let $J_{\nu}=\ZZ$ if $\mu$ parametrizes the principal or complementary series, and $J_{\nu} = [n,\infty) \cap \ZZ$ if $\mu \in \{-n^2+n\mid n \in \ZZ^{+}\}\cup\{-n^2+\frac{1}{4}\mid n \in \ZZ_{\geq 0}\}$ parametrizes the holomorphic discrete series.  Then
\begin{align*}
	X\,u_{k} &= \frac{2k + \e + 1 + \nu}{4}\cdot u_{k+1} - \frac{2k +\e  - 1 - \nu}{4}\cdot u_{k-1} \quad \textrm{for all } k \in J_{\nu}
\end{align*}
(for $\HH_\mu$ in the discrete series, the above formula must read $X\,u_{n} = (n+\frac{\e}{2})\cdot u_{n+1}$).
\end{lemma}

\begin{proof}
One calculates
\begin{align*}
	X\,u_k &= \frac{1}{2}(\eta_{+}\,u_k + \eta_{-}\,u_k) \\
		&= \frac{2k + \e +1+\nu}{4}\,u_{k+1} -\frac{2k+\e - 1 - \nu}{4}\,u_{k-1},
\end{align*}
for $k > i_\nu$, and a similar calculation for $k < i_{\nu}$ (for the principal and complementary series).  In the discrete series, with $i_{\nu} = n$, notice that $\nu = 2n +\e - 1$, so
\[
	\frac{2n+\e - 1 - \nu}{4} = 0
\]
and
\[
	\frac{2n+\e + 1 + \nu}{4} = n + \frac{\e}{2}.
\]
\end{proof}

Defining 
\[
	b^{+}(k) = \frac{2k + \e +1+\nu}{4}
\]
and 
\[
	b^{-}(k) = \frac{2k + \e -1-\nu}{4},
\]
we have simply
\begin{align*}
	X\,u_k &= b^{+}(k)\,u_{k+1} -b^{-}(k)\,u_{k-1}.
\end{align*}
Using this notation, the coboundary equation for $X$ in an irreducible unitary representation of $\SL(2,\RR)$ becomes a difference equation involving the $b^{\pm}(k)$'s and $u_k$'s.  It is exactly the difference equation found in the proof of Theorem \ref{m} for the $\PSL(2,\RR)$ case, so that the solution there works equally well here \cite{M2}.  We elaborate more on this in Appendix \ref{appendix}.

%================================================

\subsection{Representations of products} \label{repprods}

Since we are concerned in this work with products $\SL(2,\RR) \times\dots\times\SL(2,\RR):= \SL(2,\RR)^d$, we take this opportunity to set some notation for the irreducible representations.  

Irreducible unitary representations of $\SL(2,\RR)^d$ are on tensor products $\HH_{\mu_1}\otimes\dots\otimes\HH_{\mu_d}$ of Hilbert spaces of irreducible unitary representations of $\SL(2,\RR)$.  Therefore, we can use the multi-index $(\mu_1, \dots, \mu_d)$ to parametrize these.  It is sometimes convenient to use the typographical convention of taking \textbf{bold-faced} symbols to denote multi-indices corresponding to these tensor products.  For example, $\bmu:=(\mu_1,\dots,\mu_d)$.

A basis for $\HH_{\bmu}$ is given by tensor products 
\[
	\{u_{k_1}^{(1)}\otimes\dots\otimes u_{k_d}^{(d)}\}_{k_i \in J_{\nu_i}}
\]
of basis elements from the different factors.  Lemma \ref{action} gives us the following formula for the action of $X_{i}$:
\begin{align*}
	X_{i}(u_{k_1}^{(1)} \otimes\dots\otimes u_{k_d}^{(d)}) &= b_{i}^{+}(k_i)\cdot u_{k_1}^{(1)}\otimes\dots\otimes u_{k_i+1}^{(i)}\otimes\dots\otimes u_{k_d}^{(d)} \\
		&\indent - b_{i}^{-}(k_i)\cdot u_{k_1}^{(1)}\otimes\dots\otimes u_{k_i-1}^{(i)}\otimes\dots\otimes u_{k_d}^{(d)}.
\end{align*}
where
\begin{align*}
	b_{i}^{\pm}(k_{i}) &= \frac{2k_{i} + \e_i \pm (1 + \nu_{i})}{4}
\end{align*}
and $k_{i} \in J_{\nu_i}$.  We also have
\begin{align*}
	\left\| u_{k_1} \otimes\dots\otimes u_{k_d} \right\|^{2} &= \left\| u_{k_1} \right\|^{2} \dots \left\| u_{k_d} \right\|^{2} \\
		&= |\Pi_{\nu_1,\e_1,|k_1|}|\cdots|\Pi_{\nu_d, \e_d, |k_d|}| \\
		&:= |\Pi_{\bnu,\bep, {\bk}}|,
\end{align*}
where we have dropped the parenthetical superscripts, trusting that there is no confusion.

We work with elements $f \in \HH_{\bmu}$ that are of the form 
\begin{align*}
	f &= \sum_{(k_1,\dots,k_d) \in J_{\nu_1}\times\dots\times J_{\nu_d}} f(k_1,\dots,k_d)\,u_{k_1}\otimes\dots\otimes u_{k_d}.
\end{align*}
Sobolev norms are then given by
\begin{align*}
	\left\|f\right\|_{\t}^{2} &= \sum_{{\bk} \in J_{\bnu}} (1+ \left\|\bmu\right\| +2\,|{\bk} + \frac{\bep}{2}|^{2})^{\t} \, |f({\bk})|^2\,\left\|u_{\bk}\right\|^2,
\end{align*}
where we have set the multi-indices
\begin{align*}
	\bmu &:=(\mu_1,\dots,\mu_d), \\
	\left\|\bmu\right\| &:= \mu_1 + \dots + \mu_d, \\
	\bep&:= (\e_1,\dots,\e_d) \\
	J_{\bnu} &:= J_{\nu_1}\times\dots\times J_{\nu_d} \\
	{\bk}&:= (k_1,\dots,k_d) \\
	|{\bk}|^2 &:= k_{1}^{2}+\dots+k_{d}^{2}, \quad\mathrm{and} \\
	u_{\bk} &:= u_{k_1} \otimes \dots \otimes u_{k_d}.
\end{align*}
We will make use of projected versions of these elements: 
\begin{align*}
	(f\mid_{k_{j},\dots,k_{d}}) &= \sum_{i=1}^{j-1}\sum_{k_{i} \in J_{\nu_i}} f(k_1,\dots,k_d)\, \left\| u_{k_j} \right\|\cdots\left\| u_{k_d} \right\| \,u_{k_1}\otimes\dots\otimes u_{k_{j-1}} \\
		 &\in \HH_{\mu_1}\otimes\dots\otimes\HH_{\mu_{j-1}},
\end{align*}
and we compute the restricted Sobolev norm as
\begin{align*}
	&\left\|(f\mid_{k_{j},\dots,k_d})\right\|_{\t}^{2} \\
	&=\sum_{i=1}^{j-1}\sum_{k_{i} \in J_{\nu_i}} (1+\mu_1+\dots+\mu_{j-1}+2((k_{1} + \frac{\e_1}{2})^{2}+\dots+(k_{j-1}+ \frac{\e_{j-1}}{2})^{2}))^{\t}\\
		&\indent\times|f(k_1,\dots,k_{d})|^2\,\left\|u_{k_1}\right\|^{2}\cdots\left\|u_{k_{d}}\right\|^{2}.
\end{align*}
It is clear that
\begin{align*}
	\left\|(f\mid_{k_j,\dots,k_d})\right\|_{\t}^{2} &\leq \left\| f \right\|_{\t}^{2}.
\end{align*}
We will also need the observation that, for any $\t, \s \in \NN$, 
		\begin{align*}
			&\sum_{k_1 \in J_{\nu_1}}(1+\mu_1 + 2(k_{1}+ \frac{\e_1}{2})^{2})^{\t} \, \left\| (f\mid_{k_1}) \right\|_{\s}^{2} \\
			&= \sum_{k_1 \in J_{\nu_1}}(1+\mu_1 + 2(k_{1}+ \frac{\e_1}{2})^2)^{\t} \\
			&\indent\times\sum_{i=2}^{d} \sum_{k_i \in J_{\nu_i}} (1+\mu_2 +\dots+\mu_{d} + 2((k_{2}+ \frac{\e_2}{2})^{2}+\dots+(k_{d}+ \frac{\e_d}{2})^{2}))^{\s} \\
			&\indent\times |f(k_1,\dots,k_d)|^2 \, \left\| u_{k_1} \right\|^{2} \cdots \left\| u_{k_d} \right\|^{2} \\
			&\leq \sum_{i=1}^{d}\sum_{k_i \in J_{\nu_i}} (1+\left\|\bmu\right\| + 2\,|{\bk} + \frac{\bep}{2}|^2)^{\t+\s} \, |f(k,l)|^{2} \, \left\|u_{k_1}\otimes\dots\otimes u_{k_d} \right\|^2 \\
			&= \left\| f \right\|_{\t+\s}^{2}.
		\end{align*}
This uses the fact that for positive numbers $A$ and $B$, 
\[
	(1+A)^{m}(1+B)^{n} \leq (1+A+B)^{m+n}
\]
holds for all $m, n \in \NN$.

We end this section by introducing another notational convenience.  For $\HH_1 \otimes\dots\otimes\HH_d$, let 
\begin{align*}
	\bZ^d &:= \bZ_{\bmu}^{d} = \Z_{\mu_1} \times \dots \times \Z_{\mu_d} \\
		&= J_{\nu_1}\times\dots\times J_{\nu_d} - (i_{\nu_1}, \dots, i_{\nu_d}),
\end{align*}
where $\Z_{\mu_i}$ is defined by \eqref{defZ} in Section \ref{orthobasis}.  Thus, $\bZ^d$ can be thought of as a truncated $\ZZ^d$---integer $d$-tuples with non-negative entries in the components corresponding to discrete series representations.  This notation will be used in Section \ref{dfold}.

%===============================================

\subsection{Invariant distributions} \label{invariantdists}

We are interested in the $X$-invariant distributions for unitary representations of $\SL(2,\RR)$.  For a unitary representation on $\HH$, we define the following sets:
\[
\II_{X}(\HH) = \{\mathcal{D} \in \mathcal{E}^{\prime}(\HH) \mid \LLL_{X}\D = 0 \}
\]
and
\[
\II_{X}^{s}(\HH) = \{\mathcal{D} \in W^{-s}(\HH) \mid \LLL_{X}\D = 0 \}.
\]
These are exactly the $X$-invariant distributions, satisfying $\D(X\,h)=0$.

By looking at how $X$ acts in the Hilbert space $\HH_\mu$ of any irreducible unitary representation of $\SL(2,\RR)$, we see that for all $k \in J_{\nu}$, an $X$-invariant distribution $\D$ must satisfy
\begin{align} \label{relation1}
	\frac{\D(u_{k+1})}{\D(u_{k-1})} &= \frac{b^{-}(k)}{b^{+}(k)} := \beta(k).
\end{align}
(This formula also works for discrete series, where $\nu = 2n + \e -1$, because $b^{-}(n) = 0$.)

In any irreducible representation, there are at most two linearly independent distributions satisfying these conditions, obtained by alternately setting exactly one of $\D(u_{i_{\nu}})$ and $\D(u_{i_{\nu}+1})$ to $1$ and the other one to $0$.  (Recall that $i_{\nu}$ is used to denote the ``starting'' index in $J_{\nu}$, the set indexing the basis $\{u_k\}$ of $\HH_{\mu}$, and so for the principal and complementary series, $i_{\nu}=0$; for the discrete series, where $\nu = 2n + \e - 1$, we have $i_{\nu}=n$.)  We will call the resulting distributions $\D_{0}^{\HH_{\mu}}$ and $\D_{1}^{\HH_{\mu}}$, respectively.

\begin{remark*}
	Note that if $\HH_\mu$ is from the discrete series, then the distribution $\D_{1}^{\HH_{\mu}}$ is not $X$-invariant.  
\end{remark*}

Formulas for $\D_{0}^{\HH_\mu}(u_{k})$ and $\D_{1}^{\HH_\mu}(u_{k})$ are 
\begin{align}
	\D_{0}^{\HH_\mu}(u_{i_\nu + 2k}) &= \prod_{j = 1}^{|k|} \b(i_\nu + 2j - 1), \label{D0formula} \\
	\D_{1}^{\HH_\mu}(u_{i_\nu + 2k + 1}) &= \prod_{j = 1}^{|k|} \b(i_\nu + 2j). \label{D1formula}
\end{align}

Calculations from \cite{M2} show that the Sobolev order of these distributions is at most $\frac{1-\mathfrak{R}(\nu)}{2}$ for the principal and complementary series, and at most $0$ for the discrete series. Similar calculations also hold for representations of $\SL(2,\RR)^d$, and are carried out in Section \ref{invdiststensor},  in the context of tensor products of representations.

\begin{comment}
For now, we prove the following

\begin{lemma} \label{leqone}
	$|\D_{0}^{\HH_\mu}(u_k)| \leq 1$ and $|\D_{1}^{\HH_\mu}(u_k)| \leq 1$ for all $k\in J_{\nu}$.
\end{lemma}

\begin{proof}
	This follows from the fact that $|\b(k)| \leq 1$ for all $k \geq 0$ (and $k \geq n$ for discrete series), which we prove here.
	\begin{itemize}
	\item For the principal series, we have $\nu \in i\,\RR$, so
	\begin{align*}
		|\b(k)|^2 &= \left| \frac{2k + \e - 1 -\nu}{2k + \e + 1 +\nu} \right|^2 \\
			&= \frac{(2k + \e - 1)^2 + |\nu|^2}{(2k + \e + 1)^2 + |\nu|^2} \\
			&\leq 1.
	\end{align*}
	\item For the complementary series, we have $\nu \in (-1,1)\backslash\{0\}$ and $\e = 0$, so
	\begin{align*}
		|\b(k)| &= \left| \frac{2k - 1 -\nu}{2k + 1 +\nu} \right| \\
			&\leq 1.
	\end{align*}
	\item For the discrete series, we have $\nu = 2n + \e -1$, so
	\begin{align*}
		|\b(k)| &= \left| \frac{2k + \e - 1 -\nu}{2k + \e + 1 +\nu} \right| \\
			&= \left| \frac{2(k-n)}{2(k -n)+ 4n + 2\e} \right| \\
			&\leq 1.
	\end{align*}
	\end{itemize}
	This proves the lemma.
\end{proof}
\end{comment}

For now, we state the following lemma.  
\begin{lemma} \label{degreeoneblue}
For any $\nu_0 \in (0,1)$ there is a constant $C_{\nu_0}>0$ such that the inequalities
\[
	C_{\nu_0}^{-1}\, \frac{\mu^p}{\left(1 + \mu + 2\abs{k+\frac{\e}{2}}^2\right)^{1/2}} \leq \frac{\abs{\D_{\s}^{\HH_{\mu}}(u_{k})}^2}{\norm{u_{k}}^2} \leq C_{\nu_0}\, \frac{\mu^p}{\left(1 + \mu + 2\abs{k+\frac{\e}{2}}^2\right)^{1/2}}
\]
hold whenever $\D_{\s}^{\HH_{\mu}}(u_k)\neq 0$, with:
\begin{itemize} 
\item $p=1$ and $\s =0$ or $1$ in all principal series representations; 
\item $p=1/2$ and $\s=0$ or $1$ in all complementary series representations with $\abs{\nu} \leq \nu_0$; and,
\item $p=1/4$ and for $\s=0$ in all discrete series representations.
\end{itemize}
\end{lemma}

\begin{remark*}
The proof is computational.  We have left it for Section~\ref{lemmaproof} in the Appendix.  This lemma is used in the proof of Lemma~\ref{sobolevnorms}.  Its implementation is such that the numerators $\mu^p$ cancel, so that really the most important aspects of Lemma~\ref{degreeoneblue} are in the denominators, and in the constant $C_{\nu_0}>0$, which is uniform over representations where $\nu$ comes no closer to $1$ than $\nu_0$.  This condition on the parameter $\nu$ is an incarnation of the spectral gap condition found in the main theorems.
\end{remark*}

%===============================================

\subsubsection{Invariant distributions in products and their Sobolev order} \label{invdiststensor}

For representations of products $\SL(2,\RR)\times\dots\times\SL(2,\RR)$, we are interested in distributions that are invariant under the action by the subgroup of diagonal matrices.  We define the sets
\[
\II_{X_1, \dots, X_d}(\HH) = \{\mathcal{D} \in \mathcal{E}^{\prime}(\HH) \mid \LLL_{X_i}\D = 0 \quad \textrm{for all} \quad i=1,\dots,d \}
\]
and
\[
\II_{X_1, \dots, X_d}^{\t}(\HH) = \{\mathcal{D} \in W^{-\t}(\HH) \mid \LLL_{X_i}\D = 0 \quad \textrm{for all} \quad i=1,\dots,d \}.
\]

For an irreducible representation on $\HH_1 \otimes\dots\otimes\HH_d$, the $X_1,\dots,X_d$-invariant distributions are easy to describe.  They need only satisfy \eqref{relation1} in each component, and so can be taken to be products of $\D_0$ and $\D_1$ in the following sense.  Let $\bmu = (\mu_1,\dots,\mu_d)$.  We define
\begin{align*}
S(\bmu) &:= \left\{ {\bn}=(\s_1,\dots, \s_{d}) \mid \s_{i} = \begin{cases} 0 \textrm{ or } 1 \textrm{ if } \mu_i \textrm{ is not in the discrete series} \\ 0 \textrm{ if it is.} \end{cases}\right\}
\end{align*}
to be the set of $1$-$0$-vectors indexing the $X_1,\dots,X_d$-invariant distributions
\begin{align*}
	\D_{\bn}^{\HH_{\bmu}}(u_{\bk}) &:= \D_{\s_1}^{\HH_1}(u_{k_1})\cdots\D_{\s_{d}}^{\HH_{d}}(u_{k_d}),
\end{align*}
where ${\bk}=(k_1,\dots,k_d) \in J_{\nu_1}\times\dots\times J_{\nu_d}$.

The Sobolev order of $\D_{\bn}^{\HH_{\bmu}}$ is the smallest $\t \in \RR$ for which 
\[
\D_{\bn}^{\HH_{\bmu}}(f) = \sum_{{\bk} \in J_{\bnu}} f({\bk})\,\D_{\bn}^{\HH_{\bmu}}(u_{\bk})
\]
converges for every $f \in W^{\t}(\HH_{\bmu})$.  We have
\begin{align*}
	|\D_{\bn}^{\HH_{\bmu}}(f)|^2 &= \left|\sum_{{\bk} \in I_{\bmu}} f({\bk})\,\D_{\bn}^{\HH_{\bmu}}(u_{\bk})\right|^2 \\
		&\leq \sum_{{\bk}\in I_{\bmu}}(1+\left\|\bmu\right\|+2|{\bk} + \frac{\bep}{2}|^2)^{\t}\,|f(u_{\bk})|^2\,\left\| u_{\bk} \right\|^2\\
		&\indent\times \sum_{{\bk}\in I_{\bmu}}(1+\left\|\bmu\right\|+2|{\bk}+ \frac{\bep}{2}|^2)^{-\t}\,|\D_{\bn}^{\HH_{\bmu}}(u_{\bk})|^2\,\left\| u_{\bk} \right\|^{-2}
	\intertext{and by Lemma~\ref{degreeoneblue},}
		&\leq \left\| f\right\|_{\t}^{2}\cdot\sum_{{\bk}\in I_{\bmu}}(1+\left\|\bmu\right\|+2|{\bk}+ \frac{\bep}{2}|^2)^{-\t-\frac{d}{2}},
\end{align*}
which converges whenever $\t > 0$.  This shows that the Sobolev order of the distribution $\D_{\bn}^{\HH_{\bmu}}$ is at most $0$.  Therefore, the distributions $\D_{\bn}^{\HH_{\bmu}}$, where $\bn \in S(\bmu)$, form a basis for the set $\II_{X_1,\dots,X_d}^{\t}$ whenever $\t > 0$.
%===============================================
%===============================================

\subsection{Direct decompositions and spectral gaps}\label{gap}

The discussion in this section justifies our interest in irreducible representations.  This is standard, and can be found in \cite{Mau50a, Mau50b}.

Any unitary representation $\pi$ of $\SL(2,\RR)^d$ on a separable Hilbert space $\HH$ has a direct integral decomposition over a positive Stieltjes measure on $\RR$.  That is, the Hilbert space $\HH$ decomposes as
\begin{align}
	\HH &= \int_{\RR}^{\oplus} \HH_{\l} \, ds(\l) \label{decomp}
\end{align}
where the $\HH_\l$ are Hilbert spaces with unitary representations $\pi_\l$ of $\SL(2,\RR)^d$, and for every $f \in \HH$ and $g \in \SL(2,\RR)^d$,
\[
	\pi(g)f = \int_{\RR}^{\oplus} \pi_{\l}(g)f_\l \, ds(\l).
\]
That is, the operators $\pi(g)$ are decomposable with respect to \eqref{decomp}.  Furthermore, $ds$-almost every $\pi_{\l}$ is an irreducible unitary representation of $\SL(2,\RR)^d$.  

Since the $\pi(g)$ decompose with respect to \eqref{decomp}, it is then clear that so do the operators in the universal enveloping algebra of $\Sl(2,\RR)^d$.  Therefore, the decomposition \eqref{decomp} also holds for Sobolev spaces
\[
	W^{\t}(\HH) = \int_{\RR}^{\oplus} W^{\t}(\HH_{\l}) \, ds(\l),
\]
and spaces of invariant distributions
\[
	\II_{X_1,\dots,X_d}^{\t}(\HH) = \int_{\RR}^{\oplus} \II_{X_1,\dots,X_d}^{\t}(\HH_{\l}) \, ds(\l).
\]
This allows us to prove Theorems \ref{a} and \ref{b} by treating irreducible representations, and ``glueing'' solutions to the coboundary equation across this decomposition.  

For this glueing to work, we will need the representation on $\HH$ to have a \emph{spectral gap} for each Casimir operator $\Box_1, \dots, \Box_d$.  By this, we mean that there exists a number $\mu_0>0$ that is less than every non-zero Casimir parameter appearing in the irreducible sub-representations in the above direct integral decomposition.  Notationally, the representation on $\HH$ has a \emph{spectral gap} if there is a number $\mu_0 > 0$ with $\s(\Box_i) \cap (0,\mu_0) = \emptyset$ for $i=1,\dots,d$.

%===============================================

\section{Top-degree cohomology} \label{dfold}

Let us collect some of the notation we have defined so far.  Our proof of Theorem \ref{a} involves an inductive step, where we will look at an irreducible representation of the $(d+1)$-fold product $\SL(2,\RR)\times\dots\times\SL(2,\RR)$ on the Hilbert space $\HH_1 \otimes\dots\otimes\HH_{d+1}$.  It is convenient for us to use \textbf{bold-faced} letters and symbols to index the last $d$ components of this tensor product.  For example, we now have $\bmu:=(\mu_2,\dots,\mu_{d+1})$.  The rest are listed below.

%===============================================

\subsection{A collection of the notation}

\begin{itemize}
	\item $\bmu:= (\mu_2, \dots, \mu_{d+1})$ \\ 
			$\bnu:= (\nu_2, \dots, \nu_{d+1})$ \\
			$i_{\bnu}:=(i_{\nu_2}, \dots, i_{\nu_{d+1}})$ \\
			$\bep:= (\e_2,\dots,\e_{d+1})$
	\begin{itemize}
		\item These multi-indices are the parameters that define the representation on the space $\HH_{2}\otimes\dots\otimes\HH_{d+1}:= \HH_{\bmu}$.
	\end{itemize}
	\item $J_{\bnu} = J_{\nu_2}\times\dots\times J_{\nu_{d+1}} \subset \ZZ^d$ \\
			$\bZ^{d}:= \bZ_{\bmu}^{d} = \Z_{\mu_2}\times\dots\times\Z_{\mu_{d+1}}\subset \ZZ^d$
			\begin{itemize}
				\item These are indexing sets for the basis of $\HH_{\bmu}$.  The two sets are essentially the same, but shifted by $i_{\bnu}$.  
			\end{itemize}
			$\Z:= \Z_{\mu_{1}}$
	\item $\left\|\bmu\right\| :=  \mu_2 + \dots + \mu_{d+1}$
	\item ${\bl}, {\bj} \in \bZ^d$ or $J_{\bnu}$ are elements of $\ZZ^d$
	\item $|{\bl}|^2$ and $|{\bj}|^2$ denote the (squares of the) usual Euclidean norms of ${\bl}$ and ${\bj}$ in $\ZZ^d \subset \RR^d$.
	\item $v_{\bl}:= v_{l_2}\otimes\dots\otimes v_{l_{d+1}}$
	\begin{itemize}
		\item We use the letter $v$ instead of $u$ to denote the adapted basis elements (defined in Section \ref{orthobasis}) of $\HH_2, \dots, \HH_{d+1}$, hoping that this makes the computations easier to read.
	\end{itemize}
	\item $S(\bmu) := \left\{ {\bn}=(\s_2,\dots, \s_{d+1}) \mid \s_{i} = \begin{cases} 0 \textrm{ or } 1 \textrm{ if } \mu_i \textrm{ is not in the discrete series} \\ 0 \textrm{ if it is.} \end{cases}\right\}$
	\begin{itemize}
		\item This is the set of $0$-$1$-vectors that indexes the $X_2, \dots, X_{d+1}$-invariant distributions on $\HH_{\bmu}$.
	\end{itemize}
	\item $\D_{\bn}^{\HH_{\bmu}}(v_{\bl}):= \D_{\s_2}^{\HH_2}(v_{i_2})\cdots\D_{\s_{d+1}}^{\HH_{d+1}}(v_{i_{d+1}})$
	\item For $s>1$ and $d \in \NN$, let $s_{d+1} = 2^{d}s + \sum_{i=0}^{d-1}2^i(2s + d-i)$.
\end{itemize}

%=========================================

\subsection{Preparatory lemmas}

Theorem \ref{a} for $d=1$ is just Theorem \ref{m}.  We take an inductive step for $(d+1)$-fold products.  Assume that for $d$-fold products, the obstructions to solving the coboundary equation come from invariant distributions.  

Now, in an irreducible unitary representation of $\SL(2,\RR)^{d+1}$, we take an element 
\[
	f \in W^{s}(\HH_{1}\otimes\dots\otimes\HH_{d+1}):=W^{s}(\HH_{1}\otimes\HH_{\bmu}),
\]
where $s>1$ and $\bmu:= (\mu_2, \dots, \mu_{d+1})$ denotes the Casimir parameters for the irreducible representations on the Hilbert spaces $\HH_{2}, \dots, \HH_{d+1}$.  Provided that $f \in \ker \II_{X_1,\dots,X_{d+1}}$, we would like to solve the coboundary equation
\begin{align*}
	f &= X_{1}\,g_{1} + \dots + X_{d+1}\, g_{d+1}.
\end{align*}
Our strategy is to split $f$ as $f = f_1 + f_{\bmu}$, where $f_1$ is in the kernel of all $X_1$-invariant distributions, and $f_{\bmu}$ is in the kernel of all $X_2,\dots,X_{d+1}$-invariant distributions.  To this end, define for $k \in J_{\nu_1}$, ${\bl} \in \bZ^d$ and ${\bn} \in S(\bmu)$, 
\begin{align}
	f_{1}(k,i_{\bnu}+2{\bl} + {\bn}) &= \frac{m({\bl})}{\D_{\bn}^{\HH_{\bmu}}(v_{i_{\bnu}+2{\bl} + {\bn}})}\cdot \sum_{{\bj} \in \bZ^d} f(k,i_{\bnu}+2{\bj} + {\bn})\,\D_{\bn}^{\HH_{\bmu}}(v_{i_{\bnu}+2{\bj} + {\bn}}), \label{dfoldf}
\end{align}
where $m: \bZ^{d} \rightarrow \CC$ such that $\sum_{{\bl} \in \bZ^d}m({\bl}) = 1$ and $|m({\bl})|$ decreases to $0$ exponentially fast as $|{\bl}| \rightarrow \infty$.  Use \eqref{dfoldf} to define $f_{\bmu} = f - f_1$.

\begin{remark*}
It is worth emphasizing that $f_1$ is only non-zero on points of the form $(k, i_{\bnu} + 2{\bl} + {\bn})$ where ${\bn} \in S(\bmu)$.  For $0$-$1$-vectors ${\bn}$ that do \emph{not} appear in $S(\bmu)$, we have implicitly put $f_{1}(k, i_{\bnu} + 2{\bl} + {\bn}) = 0$.  
\end{remark*}

\begin{remark*}
Our splitting $f=f_1+f_{\bmu}$ is reminiscent of constructions previously used in proofs of local differentiable rigidity for higher-rank abelian actions~\cite{DK10, DK11, KW11}, where a perturbed action gives rise to an ``almost $1$-cocycle'' $\tilde\o$.  Roughly speaking, this means that there is a splitting $\tilde\o = \o + \e$ that expresses $\tilde\o$ as the sum of a $1$-cocycle $\o$ and an error $\e$, and there are tame bounds comparing norms of the cocycle $\o$ to norms of $\tilde{\o}$, and norms of the error $\e$ to norms of the exterior derivative $\di\tilde\o=\di\e$.  We will establish similar bounds controlling the Sobolev norms of $f_1$ and $f_{\bmu}$ in terms of the Sobolev norms of $f$, in Lemma~\ref{sobolevnorms}.  However, we do not need $f_1$ or $f_{\bmu}$ to satisfy any cocycle identities.  Rather, the definition~\eqref{dfoldf} is engineered so that $f_1$ lies in $\ker\II_{X_1}(\HH_1)$ whenever one projects along some $\bl \in J_{\bnu}$ (in the sense described in Section~\ref{repprods}), and $f_{\bmu}$ lies in $\ker\II_{X_2,\dots,X_{d+1}}(\HH_{\bmu})$ whenever it is projected along some $k \in J_{\nu_1}$.  This important property is proved in Lemma~\ref{bobs} and is later used in an inductive step for the proof of Theorem~\ref{a} (see Theorem~\ref{irreda}).
\end{remark*}

The following lemma establishes that our splitting $f=f_1 + f_{\bmu}$ preserves the regularity of $f$.  We will use the fact that, for two positive numbers $A$ and $B$,
\begin{align*}
	(1 + A + B) &\leq (1+A)(1+B)
\end{align*}
and
\begin{align*}
	(1+A)^{m} (1+B)^{n} \leq (1 + A + B)^{m+n}.
\end{align*}

\begin{lemma}\label{sobolevnorms}
	$f_1$ and $f_{\bmu}$ have the same Sobolev order as $f$.  Furthermore, for every $s>0$ there is a constant $C_{\nu_0, s} >0$ such that the bounds
	\[
		\norm{f_1}_s \leq C_{\nu_0, s}\,\norm{f}_{2s+d}\quad\textrm{and}\quad\norm{f_{\bmu}}_s \leq C_{\nu_0, s}\,\norm{f}_{2s+d}
	\]
	hold, provided that $f$ has Sobolev order at least $2s + d$ to begin with.  This holds with the same constant $C_{\nu_0, s}$ in any representation from the principal or discrete series, and any complementary series representation with $\abs{\nu}\leq \nu_0 <1$.
\end{lemma}

\begin{proof}
	Suppose $f \in W^{s}(\HH_1\otimes\HH_{\bmu})$ for some $s >0$, and let $\d >0$ be smaller than $s$.  For $f_1$, compute
	\begin{align*}
		\left\| f_1 \right\|_{s-\d}^{2} &= \sum_{(k,\bl) \in J_{\nu_1}\times J_{\bnu}}(1+\mu_1 + \left\|\bmu\right\| + 2(k + \frac{\e}{2})^2 + 2|{\bl} + \frac{\bep}{2}|^2)^{s-\d}\,|f_{1}(k, {\bl})|^2 \, \left\| u_{k}\otimes v_{\bl} \right\|^2 \\
			&= \sum_{{\bn} \in S(\bmu)}\sum_{\substack{k \in J_{\nu1} \\ {\bl} \in \bZ^d}}(1+\mu_1 + \left\|\bmu\right\| + 2(k+ \frac{\e}{2})^2 + 2\,|i_{\bnu}+2{\bl} + {\bn} + \frac{\bep}{2}|^2)^{s-\d} \\ 
			&\times\left|\frac{m({\bl})}{\D_{\bn}^{\HH_{\bmu}}(v_{i_{\bnu}+2{\bl} + {\bn}})}\right|^2 \, \left| \sum_{{\bj} \in \bZ^d} f(k,i_{\bnu}+2{\bj} + {\bn})\,\D_{\bn}^{\HH_{\bmu}}(v_{i_{\bnu}+2{\bj} + {\bn}})\right|^2 \, \left\| u_{k}\otimes v_{i_{\bnu}+2{\bl} + {\bn}}\right\|^{2}. \label{bcs1}
	\end{align*}
	We use the Cauchy--Schwartz inequality to bound the terms in the last line, giving	
	\begin{multline*}
			\leq \sum_{{\bn} \in S(\bmu)}\sum_{\substack{k \in J_{\nu_1} \\ {\bl}\in \bZ^d}}(1+\mu_1 + \left\|\bmu\right\| + 2(k+ \frac{\e}{2})^2 + 2\,|i_{\bnu}+2{\bl}+{\bn}+ \frac{\bep}{2}|^2)^{s-\d}\left|\frac{m({\bl})}{\D_{\bn}^{\HH_{\bmu}}(v_{i_{\bnu}+2{\bl}+{\bn}})}\right|^2 \\ 
			\times \left[ \sum_{{\bj} \in \bZ^d} (1 + \left\|\bmu\right\| + 2\,|i_{\bnu}+2{\bj}+{\bn}+ \frac{\bep}{2}|^2)^{\d}\, |f(k,i_{\bnu}+2{\bj}+{\bn})|^2 \,\left\| u_{k} \right\|^2\,\left\| v_{i_{\bnu}+2{\bj}+{\bn}} \right\|^2 \right]\\
			\times \left[\sum_{{\bj} \in \bZ^d} (1 + \left\|\bmu\right\| + 2\,|i_{\bnu}+2{\bj}+{\bn}+ \frac{\bep}{2}|^2)^{-\d}\,\frac{|\D_{\bn}^{\HH_{\bmu}}(v_{i_{\bnu}+2{\bj}+{\bn}})|^2}{\left\| u_{k} \right\|^2\,\left\| v_{i_{\bnu}+2{\bj}+{\bn}} \right\|^2}\right] \, \left\| u_{k}\otimes v_{i_{\bnu}+2{\bl}+{\bn}}\right\|^{2}.
	\end{multline*} 
		Continuing,
		\begin{multline*}
			\leq \sum_{{\bn} \in S(\bmu)}\sum_{\substack{k \in J_{\nu_1} \\ {\bl} \in \bZ^d}}(1+\mu_1 + \left\|\bmu\right\| + 2k^2 + 2\,|i_{\bnu}+2{\bl}+{\bn} + \frac{\bep}{2}|^2)^{s-\d}\,\left|\frac{m({\bl})}{\D_{\bn}^{\HH_{\bmu}}(v_{i_{\bnu}+2{\bl}+{\bn}})}\right|^2 \, \left\| (f\mid_{k}) \right\|_{\d}^{2}\\
			\times \left[\sum_{{\bj} \in \bZ^d} (1 + \left\|\bmu\right\| + 2\,|i_{\bnu}+2{\bj}+{\bn}+ \frac{\bep}{2}|^2)^{-\d}\,\frac{|\D_{\bn}^{\HH_{\bmu}}(v_{i_{\bnu}+2{\bj}+{\bn}})|^2}{\left\| v_{i_{\bnu}+2{\bj}+{\bn}} \right\|^2}\right] \, \left\| v_{i_{\bnu}+2{\bl}+{\bn}}\right\|^{2}, 
		\end{multline*}
		by the definition of $\left\|(f\mid_{k})\right\|_{\d}$.  Finally, since 
		\[
			(1+A+B) \leq (1+A)(1+B),
		\]
		we have
		\begin{multline*}
			\leq \sum_{{\bn} \in S(\bmu)}\left[\sum_{k \in J_{\nu_1}}(1+\mu_1 + 2(k+ \frac{\e}{2})^2)^{s-\d} \, \left\| (f\mid_{k}) \right\|_{\d}^{2}\right]\\
			\times\left[\sum_{{\bj} \in \bZ^d} (1 + \left\|\bmu\right\| + 2\,|i_{\bnu}+2{\bj}+{\bn}+ \frac{\bep}{2}|^2)^{-\d}\,\frac{|\D_{\bn}^{\HH_{\bmu}}(v_{i_{\bnu}+2{\bj}+{\bn}})|^2}{\left\| v_{i_{\bnu}+2{\bj}+{\bn}} \right\|^2}\right] \\
			\times \left[\sum_{{\bl} \in \bZ^d}(1+\left\|\bmu\right\| + 2\,|i_{\bnu}+2{\bl}+{\bn} + \frac{\bep}{2}|^2)^{s-\d}\,\left|\frac{m({\bl})}{\D_{\bn}^{\HH_{\bmu}}(v_{i_{\bnu}+2{\bl}+{\bn}})}\right|^2 \, \left\| v_{i_{\bnu}+2{\bl}+{\bn}}\right\|^2\right].
		\end{multline*}
		Using, from Section \ref{repprods}, that
		\[
			\sum_{k \in J_{\nu_1}}(1+\mu_1 + 2(k+ \frac{\e}{2})^2)^{s-\d} \, \left\| (f\mid_{k}) \right\|_{\d}^{2} \leq \left\| f \right\|_{s-\d + \d}^{2} = \left\| f \right\|_{s}^{2},
		\]
		we have
	\begin{multline*}
		\left\| f_1 \right\|_{s-\d}^{2} \leq \sum_{{\bn} \in S(\bmu)}\left[\sum_{{\bj} \in \bZ^d} (1 + \left\|\bmu\right\| + 2\,|i_{\bnu}+2{\bj}+{\bn}+ \frac{\bep}{2}|^2)^{-\d}\,\frac{|\D_{\bn}^{\HH_{\bmu}}(v_{i_{\bnu}+2{\bj}+{\bn}})|^2}{\left\| v_{i_{\bnu}+2{\bj}+{\bn}} \right\|^2}\right] \\
			\times \left[\sum_{{\bl} \in \bZ^d}(1+\left\|\bmu\right\| + 2\,|i_{\bnu}+2{\bl}+{\bn}+ \frac{\bep}{2}|^2)^{s-\d}\,\left|\frac{m({\bl})}{\D_{\bn}^{\HH_{\bmu}}(v_{i_{\bnu}+2{\bl}+{\bn}})}\right|^2 \, \left\| v_{i_{\bnu}+2{\bl}+{\bn}}\right\|^2\right]\left\| f \right\|_{s}^{2}. 
	\end{multline*}
	By Lemma \ref{degreeoneblue}, and because $\d > 0$, the term in the first line is finite.  The term in the second line is finite because $m({\bl})$ decays exponentially.  This proves that $f_{1} \in W^{s-\d}(\HH_{1}\otimes\HH_{\bmu})$, and by taking arbitrarily small $\d >0$, we see that $f_1$ must have the same Sobolev order as $f$, which immediately implies that $f_{\bmu}$ also has the same Sobolev order.
	
	For the norm estimates, we carry out a very similar calculation, arriving at
	\begin{multline*}
	\norm{f_1}_s^2 \leq \sum_{{\bn} \in S(\bmu)}\left[\sum_{{\bj} \in \bZ^d} (1 + \left\|\bmu\right\| + 2\,|i_{\bnu}+2{\bj}+{\bn}+ \frac{\bep}{2}|^2)^{-(s+d)}\,\frac{|\D_{\bn}^{\HH_{\bmu}}(v_{i_{\bnu}+2{\bj}+{\bn}})|^2}{\left\| v_{i_{\bnu}+2{\bj}+{\bn}} \right\|^2}\right] \\
			\times \left[\sum_{{\bl} \in \bZ^d}(1+\left\|\bmu\right\| + 2\,|i_{\bnu}+2{\bl}+{\bn}+ \frac{\bep}{2}|^2)^{s}\,\left|\frac{m({\bl})}{\D_{\bn}^{\HH_{\bmu}}(v_{i_{\bnu}+2{\bl}+{\bn}})}\right|^2 \, \left\| v_{i_{\bnu}+2{\bl}+{\bn}}\right\|^2\right] \left\| f \right\|_{2s+d}^{2}. 
	\end{multline*}
	Lemma~\ref{degreeoneblue} and the observation that 
	\[
		(1+A+B)\leq(1+A)(1+B)\leq(1+A+B)^2
	\]
	imply that there is some constant $C_{\nu_0} >0$ such that
	\begin{multline}\label{firstsecond}
	\norm{f_1}_s^2 \leq C_{\nu_0}\,\sum_{{\bn} \in S(\bmu)}\left[\sum_{{\bj} \in \bZ^d} (1 + \left\|\bmu\right\| + 2\,|i_{\bnu}+2{\bj}+{\bn}+ \frac{\bep}{2}|^2)^{-(s+d)-\frac{d}{2}}\right] \\
			\times \left[\sum_{{\bl} \in \bZ^d}(1+\left\|\bmu\right\| + 2\,|i_{\bnu}+2{\bl}+{\bn}+ \frac{\bep}{2}|^2)^{s+d}\,\abs{m({\bl})}^2\right] \left\| f \right\|_{2s+d}^{2}. 
	\end{multline}
	The last term in~\eqref{firstsecond} is bounded by
	\begin{multline*}
	\sum_{{\bl} \in \bZ^d}(1+\left\|\bmu\right\| + 2\,|i_{\bnu}+2{\bl}+{\bn}+ \frac{\bep}{2}|^2)^{s+d}\,\abs{m({\bl})}^2 \\
	\leq (1+\left\|\bmu\right\|)^{s+d}\left[\sum_{{\bl} \in \bZ^d}(1+ 2\,|i_{\bnu}+2{\bl}+{\bn}+ \frac{\bep}{2}|^2)^{s+d}\,\abs{m({\bl})}^2\right]
	\end{multline*}
	and the term in brackets converges and is bounded by some constant $C_s>0$ that only depends on $s$.  We therefore have,
	\[
		\leq C_{s} (1+\left\|\bmu\right\|)^{s+d}.
	\]
	The first term in~\eqref{firstsecond} is bounded by an integral that is bounded by
	\[
		\leq \frac{C_s}{(1 + \norm{\bmu})^{s+d}},
	\]
	where $C_s>0$ is some other constant that only depends on $s$.  Combining the constants, we have shown that there exists $C_{\nu_0, s}>0$ such that 
	\begin{align*}
		\norm{f_1}_{s}^2 \leq C_{\nu_0, s}\,\norm{f}_{2s+d}^2
	\end{align*}
	as desired.  Since $f_{\bmu} = f - f_1$, we can just replace the above constant $C_{\nu_0, s}$ with $C_{\nu_0, s}+1$, and use the triangle inequality to get 
	\begin{align*}
		\norm{f_{\bmu}}_{s}^2 \leq C_{\nu_0, s}\,\norm{f}_{2s+d}^2.
	\end{align*}
	The constant only depends on $s$, and is uniform over all principal and discrete series representations, and all complementary series representations with $\abs{\nu} \leq \nu_0$, where $\nu_0 \in (0,1)$.  This completes the proof of the Lemma.
\end{proof}

The following lemma is then automatic.

\begin{lemma} \label{breg}
	If $f \in W^{s}(\HH_{1}\otimes\HH_{\bmu})$ with $s>0$, then $(f_{1}\mid_{\bl}) \in W^{s}(\HH_{1})$ for all ${\bl} \in J_{\bnu}$ and $(f_{\bmu}\mid_{k}) \in W^{s}(\HH_{\bmu})$ for all $k \in J_{\nu_1}$.
\end{lemma}

The next lemma shows that $(f_{1}\mid_{\bl})$ is in the kernel of all $X_1$-invariant distributions on $\HH_{1}$, and that $(f_{\bmu}\mid_{k})$ is in the kernel of all $X_2, \dots, X_{d+1}$-invariant distributions on $\HH_{\bmu}$.

\begin{lemma} \label{bobs}
	For every ${\bl} \in J_{\bnu}$, we have $(f_{1}\mid_{\bl}) \in \ker\II_{X_1}^{s}(\HH_{1})$.  Similarly, for every $k \in J_{\nu_1}$, we have $(f_{\bmu}\mid_{k}) \in \ker\II_{X_2,\dots,X_{d+1}}^{s}(\HH_{\bmu})$.
\end{lemma}

\begin{proof}
We do the calculations for $\D_{0}^{\HH_{1}}$, $\D_{1}^{\HH_{1}}$, and $\D_{\bn}^{\HH_{\bmu}}$.  First, for ${\bl} \in \bZ^d$ and any ${\bn} \in S(\bmu)$,
\begin{align*}
	&\D_{0}^{\HH_1}(f_{1}\mid_{i_{\bnu}+2{\bl} + {\bn}}) \\
		&\indent = \sum_{k \in \Z} f_{1}(i_{\nu_1}+2k,i_{\bnu}+2{\bl} + {\bn})\,\left\| v_{i_{\bnu}+2{\bl} + {\bn}} \right\| \, \D_{0}^{\HH_1}(u_{i_{\nu_1}+2k}) \\
		&\indent = \sum_{k \in \Z} \, \left[\frac{m({\bl})}{\D_{\bn}^{\HH_{\bmu}}(v_{i_{\bnu}+2{\bl} + {\bn}})}\cdot \sum_{{\bj} \in \bZ^d} f(i_{\nu_1}+2k,i_{\bnu}+2{\bj} + {\bn})\,\D_{\bn}^{\HH_{\bmu}}(v_{i_{\bnu}+2{\bj} + {\bn}})\right]\\
		&\indent\indent\times\left\| v_{i_{\bnu}+2{\bl} + {\bn}} \right\| \, \D_{0}^{\HH_1}(u_{i_{\nu_1}+2k}) \\
		&\indent = \frac{m({\bl})\, \left\| v_{i_{\bnu}+2{\bl} + {\bn}} \right\|}{\D_{\bn}^{\HH_{\bmu}}(v_{i_{\bnu}+2{\bl} + {\bn}})}\cdot \sum_{k \in \Z} \, \sum_{{\bj} \in \bZ^d} f(i_{\nu_1}+2k,i_{\bnu}+2{\bj} + {\bn})\,\D_{\bn}^{\HH_{\bmu}}(v_{i_{\bnu}+2{\bj} + {\bn}})\,\D_{0}^{\HH_1}(u_{i_{\nu_1}+2k}) \\
		&\indent = \frac{m({\bl})\, \left\| v_{i_{\bnu}+2{\bl} + {\bn}} \right\|}{\D_{\bn}^{\HH_{\bmu}}(v_{i_{\bnu}+2{\bl} + {\bn}})}\cdot \D_{0,\bn}^{\HH_{1} \otimes \HH_{\bmu}}(f) \\
		&\indent = 0,
\end{align*}
and, if $\HH_1$ is not from the discrete series,
\begin{align*}
	&\D_{1}^{\HH_1}(f_{1}\mid_{i_{\bnu}+2{\bl} + {\bn}}) \\
		&\indent = \sum_{k \in \Z} f_{1}(i_{\nu_1}+2k+1,i_{\bnu}+2{\bl} + {\bn})\,\left\| v_{i_{\bnu}+2{\bl} + {\bn}} \right\| \, \D_{1}^{\HH_1}(u_{i_{\nu_1}+2k+1}) \\
		&\indent = \sum_{k \in \Z} \, \left[\frac{m({\bl})}{\D_{\bn}^{\HH_{\bmu}}(v_{i_{\bnu}+2{\bl} + {\bn}})}\cdot \sum_{{\bj} \in \bZ^d} f(i_{\nu_1}+2k+1,i_{\bnu}+2{\bj} + {\bn})\,\D_{\bn}^{\HH_{\bmu}}(v_{i_{\bnu}+2{\bj} + {\bn}})\right]\\
		&\indent\indent\times\left\| v_{i_{\bnu}+2{\bl} + {\bn}} \right\| \, \D_{1}^{\HH_1}(u_{i_{\nu_1}+2k+1}) \\
		&\indent = \frac{m({\bl})\, \left\| v_{i_{\bnu}+2{\bl} + {\bn}} \right\|}{\D_{\bn}^{\HH_{\bmu}}(v_{i_{\bnu}+2{\bl} + {\bn}})}\cdot \sum_{k \in \Z} \, \sum_{{\bj} \in \bZ^d} f(i_{\nu_1}+2k+1,i_{\bnu}+2{\bj} + {\bn})\,\D_{\bn}^{\HH_{\bmu}}(v_{i_{\bnu}+2{\bj} + {\bn}})\,\D_{1}^{\HH_1}(u_{i_{\nu_1}+2k+1}) \\
		&\indent = \frac{m({\bl})\, \left\| v_{i_{\bnu}+2{\bl} + {\bn}} \right\|}{\D_{\bn}^{\HH_{\bmu}}(v_{i_{\bnu}+2{\bl} + {\bn}})}\cdot \D_{1,\bn}^{\HH_{1}\otimes\HH_{\bmu}}(f) \\
		&\indent = 0.
\end{align*}
We have just shown, for every ${\bl} \in J_{\bnu}$, that $(f_{1}\mid_{\bl}) \in W^{s}(\HH_{1})$ is in the kernel of every $X_1$-invariant distribution.

For $f_{\bmu}$, the calculations are somewhat quicker.  For any ${\bn} \in S(\bmu)$,
\begin{align*}
	&\D_{\bn}^{\HH_{\bmu}}(f_{\bmu}\mid_{k})  \\
		&= \sum_{{\bl} \in \bZ^d} f_{\bmu}(k, i_{\bnu}+2{\bl}+{\bn})\,\left\| u_{k} \right\| \, \D_{\bn}^{\HH_{\bmu}}(v_{i_{\bnu}+2{\bl}+{\bn}}) \\
		&= \sum_{{\bl} \in \bZ^d} \left[f(k,i_{\bnu}+2{\bl}+{\bn}) - \frac{m({\bl})}{\D_{\bn}^{\HH_{\bmu}}(v_{i_{\bnu}+2{\bl}+{\bn}})}\cdot \sum_{{\bj} \in \bZ^d} f(k,i_{\bnu}+2{\bj}+{\bn})\,\D_{\bn}^{\HH_{\bmu}}(v_{i_{\bnu}+2{\bj}+{\bn}})\right]\\
		&\indent\times\left\| u_{k} \right\| \, \D_{\bn}^{\HH_{\bmu}}(v_{i_{\bnu}+2{\bl}+{\bn}}) \\
		&= \sum_{{\bl} \in \bZ^d} f(k,i_{\bnu}+2{\bl}+{\bn})\,\left\| u_{k} \right\| \, \D_{\bn}^{\HH_{\bmu}}(v_{i_{\bnu}+2{\bl}+{\bn}}) \\
		&\indent\indent\indent- \sum_{{\bl} \in \bZ^d}m({\bl})\cdot \sum_{{\bj} \in \bZ^d} f(k,i_{\bnu}+2{\bj}+{\bn})\,\left\| u_{k} \right\| \, \D_{\bn}^{\HH_{\bmu}}(v_{i_{\bnu}+2{\bj}+{\bn}}) \\
		&= 0,
\end{align*}
proving the lemma.
\end{proof}

We are now prepared to state the proof of a version of Theorem \ref{a} for irreducible representations, from which will follow Theorem \ref{a}. 

%===============================================

\subsection{Irreducible case}

\begin{theorem} \label{irreda}
	Let $\HH = \HH_{1}\otimes\dots\otimes\HH_{d}$ be the Hilbert space of an irreducible unitary representation of $\SL(2,\RR)^d$, and let
	\[
	\min\{\mu_1, \dots, \mu_{d}\} > \mu_0 > 0.
	\]
	Then, for every $s > 1$ and $t < s-1$, there is a constant $C_{\mu_0,s,t}$ such that, for every $f \in \ker \II_{X_1,\dots,X_d}^{s}(\HH)$, there exist $g_1,\dots, g_d \in W^{t}(\HH)$ satisfying the degree-$d$ coboundary equation \eqref{cobeq} for $f$, and satisfying the Sobolev estimates
	\begin{align*} 
		\left\| g_{i} \right\|_{t} &\leq C_{\mu_0,s,t} \,\left\| f \right\|_{s_d}
	\end{align*}  
	for $i=1, \dots, d$.
\end{theorem}

\begin{remark*}
The Sobolev norm bound in Theorem~\ref{irreda} only gives non-trivial information when $f \in W^{s_d}(\HH)$.  
\end{remark*}

\begin{proof}
The theorem is known to hold for $d=1$, by Theorem \ref{m}.  Assume the theorem holds for $d$-fold products.  Our proof consists of showing that this implies the theorem for $(d+1)$-fold products.  We keep the notations used in Section \ref{dfold}.

We are given $f \in \ker \II_{X_1,\dots,X_{d+1}}^{s}(\HH_{1}\otimes\HH_{\bmu})$, with $s > 1$.  We define $f_{1}$ and $f_{\bmu}$ as in \eqref{dfoldf}.  By Lemmas \ref{breg} and \ref{bobs}, we have that $(f_{1}\mid_{\bl}) \in \ker \II_{X_1}^{s}(\HH_{1})$ for all ${\bl} \in J_{\bnu}$, and $(f_{\bmu}\mid_{k}) \in \ker\II_{X_2,\dots,X_{d+1}}^{s}(\HH_{\bmu})$ for all $k \in J_{\nu_1}$.

By the inductive hypothesis, for every ${\bl} \in J_{\bnu}$, there exists $g_{\bl} \in W^{t}(\HH_{1})$ satisfying $X_{1}\,g_{\bl} = (f_{1}\mid_{\bl})$, where $t<s-1$.  Similarly, for every $k \in J_{\nu_1}$, there exists $h_{2,k},\dots,h_{d+1,k} \in W^{t}(\HH_{\bmu})$ satisfying 
\[
	X_{2}\,h_{2,k} + \dots + X_{d+1}\,h_{d+1,k} = (f_{\bmu}\mid_{k}),
\] 
where $t<s-1$.  Also, we have the estimates
\begin{align*}
	\left\| g_{\bl} \right\|_{t} &\leq C_{\mu_0,s,t}\,\left\| (f_{1}\mid_{\bl}) \right\|_{s_{1}=s} \\
		&\mathrm{and} \\
	\left\| h_{i,k} \right\|_{t} &\leq C_{\mu_0,s,t}\,\left\| (f_{\bmu}\mid_{k}) \right\|_{s_d}
\end{align*}
for all $i = 2, \dots, d+1$, where $0<\mu_0< \min\{\mu_1,\dots, \mu_d\}$, and $C_{\mu_0, s, t}>0$ is the constant from Theorem \ref{m}.

Setting
\begin{align*}
	g_{1}(k,{\bl}) &= \frac{g_{\bl}(k)}{\left\| v_{\bl} \right\|} \\
	&\mathrm{and}\\
	g_{i}(k,{\bl}) &= \frac{h_{i,k}({\bl})}{\left\| u_{k} \right\|}
\end{align*}
for $i = 2, \dots, d+1$, we have that
\begin{align*}
	\sum_{i=1}^{d+1} X_i\,g_{i} &= \sum_{i=1}^{d+1} X_i \,\sum_{(k,{\bl}) \in J_{\nu_1} \times J_{\bnu}} g_{i}(k,{\bl})\,u_{k}\otimes v_{\bl} \\
		&= \sum_{i=1}^{d+1} \sum_{(k,{\bl}) \in J_{\nu_1} \times J_{\bnu}} g_{i}(k,{\bl})\,(X_i\,u_{k}\otimes v_{\bl}) \\
		&= \sum_{{\bl} \in J_{\bnu}}\,\left\| v_{\bl} \right\|^{-1}\,(X_1\, g_{\bl})\otimes v_{\bl} + \sum_{i=2}^{d+1}\sum_{k \in J_{\nu_1}}\left\| u_{k} \right\|^{-1}\, u_{k}\otimes(X_i \, h_{i,k}) \\
		&= \sum_{{\bl} \in J_{\bnu}}\left\| v_{\bl} \right\|^{-1}\,(f_{1}\mid_{\bl})\otimes v_{\bl} + \sum_{i=2}^{d+1}\sum_{k \in J_{\nu_1}} \left\| u_{k} \right\|^{-1}\,u_{k}\otimes(f_{\bmu}\mid_{k}) \\
		&= f,
\end{align*} 
and so $g_1, \dots, g_{d+1}$ constitute a formal solution to the coboundary equation.  To see that it is a \emph{bona fide} solution, we just check the Sobolev norms, 
\begin{align*}
	\left\| g_1 \right\|_{t}^{2} &= \sum_{(k,{\bl}) \in J_{\nu_1} \times J_{\bnu}} (1+\mu_1 + \left\|\bmu\right\| + 2(k+ \frac{\e}{2})^2+2\,|{\bl} + \frac{\bep}{2}|^2)^t\,|g_{1}(k,{\bl})|^2\,\left\| u_{k}\otimes v_{\bl} \right\|^2 \\
		&\leq \sum_{{\bl} \in J_{\bnu}}\,\sum_{k \in J_{\nu_1}} (1+\left\|\bmu\right\| + 2\,|{\bl}+ \frac{\bep}{2}|^2)^t\,(1+\mu_1 + 2(k+ \frac{\e}{2})^2)^t\,|g_{1}(k,{\bl})|^2\,\left\| u_{k}\right\|^2 \,\left\| v_{\bl} \right\|^2 \\
		&= \sum_{{\bl} \in J_{\bnu}} (1+\left\|\bmu\right\| + 2\,|{\bl}+ \frac{\bep}{2}|^2)^t\,\left\|g_{\bl} \right\|_{t}^{2} \\
		&\leq C_{\mu_0, s,t}^2\,\sum_{{\bl} \in J_{\bnu}} (1+\left\|\bmu\right\| + 2\,|{\bl}+ \frac{\bep}{2}|^2)^{s}\,\left\| (f_{1}\mid_{\bl}) \right\|_{s}^{2} \\
		&\leq C_{\mu_0,s,t}^2\,\left\| f_1 \right\|_{2s}^2 \leq C_{\mu_0,s,t}^2\,\norm{f}_{4s+d}^2 \leq C_{\mu_0,s,t}^2\,\norm{f}_{s_{d+1}}^2,
\end{align*}
where in the last line we have absorbed the constant $C_{\nu_0, s}$ from Lemma~\ref{sobolevnorms} into $C_{\mu_0,s,t}$ and also observed that $4s+d \leq s_{d+1}$.  For $i=2,\dots, d+1$,
\begin{align*}
	\left\| g_i \right\|_{t}^{2} &= \sum_{(k,{\bl}) \in J_{\nu_1} \times J_{\bnu}} (1+\mu_1 + \left\|\bmu\right\| + 2(k+ \frac{\e}{2})^2+2\,|{\bl}+ \frac{\bep}{2}|^2)^t\,|g_{i}(k,{\bl})|^2\,\left\| u_{k}\otimes v_{\bl} \right\|^2 \\
		&\leq \sum_{k \in J_{\nu_1}}\,\sum_{{\bl} \in J_{\bnu}} (1+\left\|\bmu\right\| + 2\,|{\bl}+ \frac{\bep}{2}|^2)^t\,(1+\mu_1 + 2(k+ \frac{\e}{2})^2)^t\,|g_{i}(k,{\bl})|^2\,\left\| v_{\bl} \right\|^2 \, \left\| u_{k}\right\|^2 \\
		&= \sum_{k \in J_{\nu_1}} (1+\mu_1 + 2(k+ \frac{\e}{2})^2)^t\,\left\|h_{i,k} \right\|_{t}^{2} \\
		&\leq C_{\mu_0,s,t}^2\,\sum_{k \in J_{\nu_1}} (1+\mu_1 + 2(k+ \frac{\e}{2})^2)^{s}\,\left\| (f_{\bmu}\mid_{k}) \right\|_{s_d}^{2} \\
		&\leq C_{\mu_0,s,t}^2\,\left\| f \right\|_{2(s_{d}+s)+d}^2 = C_{\mu_0,s,t}^2\,\left\| f \right\|_{s_{d+1}}^2,
\end{align*}
where we have again used Lemma~\ref{sobolevnorms} and absorbed its constant into $C_{\mu_0,s,t}$, and noted that the recursion $s_{d+1} = 2(s_d + s)+d$ starting with $s_1 = s$ results in the formula
\[
	s_{d+1}=2^{d}s + \sum_{i=0}^{d-1} 2^i(2s + d-i).
\]
This proves Theorem \ref{a} in the irreducible case.
\end{proof}

%===============================================

\subsection{Proof of Theorem \ref{a}}

\begin{proof}
We have a unitary representation of $\SL(2,\RR)^d$ on $\HH$ with spectral gap, as in the theorem statement.  Let $s,t, \mu_0$ be as in the theorem statement.  Consider the direct integral decomposition
\[
	\HH = \int_{\oplus}\HH_{\l}\,ds(\l)
\]
where $ds$-almost all $\HH_{\l}$ are irreducible.  Also, for all $s \in \RR$,
\[
	W^{s}(\HH) = \int_{\oplus} W^{s}(\HH_{\l})\,ds(\l).
\]
Any $f \in \ker\II_{X_1,\dots,X_d}^{s_d}(\HH)$ decomposes as
\[
	f = \int_{\oplus} f_{\l}\,ds(\l),
\]
where $f_{\l} \in W^{s_d}(\HH_{\l})$.  Since invariant distributions also decompose
\[
	\II_{X_1,\dots,X_d}^{s_d}(\HH) = \int_{\oplus} \II_{X_1,\dots,X_d}^{s_d}(\HH_{\l})\,ds(\l),
\]
we have that, for $ds$-almost every $\l$, 
\[
f_{\l} \in \ker\II_{X_1,\dots,X_d}^{s_d}(\HH_{\l}) \subset \ker\II_{X_1,\dots,X_d}^{s}(\HH_{\l}),
\]
and so by Theorem \ref{irreda}, there are $g_{1,\bmu},\dots,g_{d,\bmu} \in W^{t}(\HH_{\bmu})$ satisfying the coboundary equation, and the estimate
\[
	\left\| g_{i,\bmu}\right\|_{t} \leq C_{\mu_0,s,t} \left\| f_{\bmu}\right\|_{s_d}
\]
in $ds$-almost every irreducible $\HH_{\bmu}$ appearing in the decomposition.  Set
\begin{align*}
	g_{i} &:= \int_{\oplus} g_{i, \l}\,ds(\l),
\end{align*}
where $g_{i,\l} = g_{i,\bmu}$ for all $\l$ where $f_\l \in \ker\II_{X_1,\dots,X_d}^{s_d}(\HH_{\l})$ and $\HH_{\l} = \HH_{\bmu}$ is irreducible, and $g_{i,\l} = 0$ otherwise.   Then
\begin{align*}
	\left\| g_{i}\right\|_{t}^{2} &= \int_{\oplus} \left\|g_{i, \l}\right\|_{t}^2 \,ds(\l)\\
		&\leq C_{\mu_0,s,t}^{2}\int_{\oplus} \left\|f_{\l}\right\|_{s_d}^{2}\,ds(\l) \\
		&= C_{\mu_0,s,t}^{2} \left\|f \right\|_{s_d}^{2}.
\end{align*}
The vectors $g_1,\dots,g_d$ constitute a solution to the coboundary equation because the operators $X_1,\dots,X_d$ are decomposable with respect to the direct integral decomposition.  This completes the proof of Theorem \ref{a}.
\end{proof}

%===============================================

\section{Intermediate cohomology}

Continuing with our usual notation, let $\HH_1 \otimes\dots\otimes\HH_d$ be the Hilbert space of an irreducible unitary representation of the group $\SL(2,\RR)^d$, with all factors non-trivial.  We define an \emph{$n$-form (of Sobolev order at least $\t$) over the $\RR^d$-action on $\HH_{1}\otimes\dots\otimes\HH_d$} to be a map 
\[
\o:(\mathrm{Lie}(\RR^{d}))^n \rightarrow W^{\t}(\HH_{1}\otimes\dots\otimes\HH_d)
\] 
which is linear and anti-symmetric.  There is an exterior derivative, given by the formula
\begin{align*}
	\di\o(V_{1},\dots,V_{n+1}) &:= \sum_{j=1}^{n+1} (-1)^{j+1}\,V_{j}\,\o(V_1,\dots,\widehat{V_{j}},\dots, V_{n+1}),
\end{align*}  
where ``$\quad \widehat{}\quad$'' denotes omission.  One sees that $\di\o$ is an $(n+1)$-form taking values in a lower Sobolev space $W^{\t-1}(\HH_1 \otimes\dots\otimes\HH_d)$.

One can see $\o$ as an element $\o \in W^{\t}(\HH_{1}\otimes\dots\otimes\HH_d)^{\binom{d}{n}}$, indexed by $n$-tuples from the set $\left\{X_1, \dots, X_d\right\}$.  The form $\o$ is said to be \emph{closed}, and is called a \emph{cocycle}, if $\di\o=0$, or
\begin{align*}
	\di\o(X_{I}) &:= \sum_{j=1}^{n+1} (-1)^{j+1}\,X_{i_{j}}\,\o(X_{I_{j}}) = 0,
\end{align*}  
where $I:=(i_{1},\dots,i_{n+1})$ with $i_j \in \left\{1,\dots,d \right\}$ is the multi-index, and
\[
I_{j}:=(i_1,\dots, \widehat{i_{j}}, \dots, i_n).
\]
It is \emph{exact}, and is called a \emph{coboundary}, if there is an $(n-1)$-form $\eta$ satisfying $\di\eta=\o$.  Two forms that differ by a coboundary are said to be \emph{cohomologous}.  We denote the space of $n$-forms over the $\RR^d$-action on $\HH_1 \otimes\dots\otimes\HH_d$ by $\O_{\RR^d}^{n}(W^{\t}(\HH_{1}\otimes\dots\otimes\HH_d))$.  (These mirror the usual definitions from de Rham cohomology.)

Notice that if $n=d$, then $\o$ is given by just one element 
\[
	\o(X_1,\dots,X_d) = f \in W^{\t}(\HH_1 \otimes\dots\otimes\HH_d); 
\]
it is automatically closed, and exactness is characterized by the existence of a $(d-1)$-form $\eta$ satisfying $\di\eta=\o$.  Or, setting $\eta(X_{I_j}) = (-1)^{j+1}g_{j}$,
\begin{align*}
	\di\eta(X_1,\dots,X_d) &= \sum_{j=1}^{d} (-1)^{j+1}\,X_{j}\,\eta(X_{I_j}) \\
			&= \sum_{j=1}^{d} X_{j}\,g_j \\
			&= f.
\end{align*}
This is exactly the top-degree coboundary equation \eqref{cobeq} from the first part of this paper. 

It will be useful to define restricted versions of forms.  For an $n$-form $\o \in \O_{\RR^d}^{n}(W^{\t}(\HH_{1}\otimes\dots\otimes\HH_d))$, define $\o_1 \in \O_{\RR^{d}}^{n}(W^{\t}(\HH_{1}\otimes\dots\otimes\HH_d))$ to be indexed by $2\leq i_1 < \dots < i_n \leq d$,
\[
	\o_{1}(X_{i_1},\dots,X_{i_n}) = \o(X_{i_1},\dots,X_{i_n}).
\]
This is just the form $\o$, with the index $1$ ``missing.''  Fixing a basis element $u_{k} \in \HH_1$, we can define a restricted version $(\o_1 \mid_{k}) \in\O_{\RR^{d-1}}^{n}(W^{\t}(\HH_{2}\otimes\dots\otimes\HH_d))$ by
\[
	(\o_{1}\mid_{k})(X_{i_1},\dots,X_{i_n}) = (\o(X_{i_1},\dots,X_{i_n}))\mid_{k}.
\]
This is an $n$-form over the $\RR^{d-1}$-action by $X_2,\dots,X_d$ on $\HH_2 \otimes\dots\otimes\HH_d$.  We prove the following lemma, which shows that if $\o$ is a closed form, then so are $\o_1$ and $(\o_{1}\mid_{k})$.

\begin{lemma} \label{closed}
	Let $\o \in \O_{\RR^d}^{n}(W^{\t}(\HH_{1}\otimes\dots\otimes\HH_{d}))$, with $\di\o=0$.  Then $\di\o_1=0$ and $\di(\o_{1}\mid_{k})=0$ for all $k \in J_{\nu_1}$.
\end{lemma}

\begin{proof}
These are calculations.  First, for $i_1 \geq 2$, 
\begin{align*}
	\di\o_1 (X_{i_1},\dots,X_{i_{n+1}}) &= \sum_{j=1}^{n+1}(-1)^{j+1}\,X_{i_j}\,\o_{1}(X_{i_1},\dots,\widehat{X_{i_j}},\dots,X_{i_{n+1}}) \\
		&= \sum_{j=1}^{n+1}(-1)^{j+1}\,X_{i_j}\,\o(X_{i_1},\dots,\widehat{X_{i_j}},\dots,X_{i_{n+1}}) \\
		&= \di\o(X_{i_1},\dots,X_{i_{n+1}}) \\
		&= 0.
\end{align*}
The calculation for $\di(\o_{1}\mid_{k})$ is equally straight-forward.
\end{proof}

Closed $(d-1)$-forms over the $\RR^d$-action on $\HH_1 \otimes\dots\otimes\HH_d$ are of special interest here.  The following proposition shows that for $\o \in \O_{\RR^d}^{d-1}(W^{\t}(\HH_1 \otimes\dots\otimes\HH_d))$ with $\di\o=0$, the top-degree cocycle 
\[
(\o_{1}\mid_{k}) \in \O_{\RR^{d-1}}^{d-1}(W^{\t}(\HH_2 \otimes\dots\otimes\HH_{d}))
\]
is in the kernel of $X_2,\dots,X_d$-invariant distributions, and hence is exact, for every $k\in J_{\nu_1}$, by Theorem \ref{a}.

\begin{proposition} \label{prop}
	Let $\o \in \O_{\RR^{d}}^{d-1}(W^{\t}(\HH_{1}\otimes\dots\otimes\HH_d))$ be a closed $(d-1)$-form.  Then for every $k \in J_{\nu_1}$, we have that 
	\[
	(\o_{1}\mid_{k})(X_2,\dots,X_d) \in \ker\II_{X_2,\dots,X_{d}}^{\t}.
	\]
\end{proposition}

\begin{proof}
	Setting
	\[
		f_j:= (-1)^{j+1}\,\o(X_1,\dots,\widehat{X_j},\dots,X_d),
	\] 
	we have that
	\[
		X_1\,f_1 + \dots + X_d\,f_d = 0,
	\]
	and $f_{1}\mid_{k} = (\o_{1}\mid_{k})(X_2,\dots,X_d) \in W^{\t}(\HH_2 \otimes\dots\otimes\HH_d)$.  
	
	Suppose $\D \in W^{-\t}(\HH_2 \otimes\dots\otimes\HH_d)$ is $X_2,\dots,X_d$-invariant.  Define the map
	\begin{align*}
		\bar{\D}: W^{\t}(\HH_1 \otimes\dots\otimes\HH_{d}) &\rightarrow W^{\t}(\HH_1) \\
			\mathrm{by} \quad u_{k}\otimes v_{\bl}: &\mapsto \D(v_{\bl})\cdot u_{k},
	\end{align*}
	and extending linearly.  Then $\bar\D$ is also $X_2,\dots,X_d$-invariant, in the sense that $\bar\D(X_i\,h) = 0$ for all $h \in W^{\t}(\HH_1 \otimes\dots\otimes\HH_d)$ and $i=2,\dots,d$.  Therefore,
	\begin{align*}
		\bar{\D}(X_1\,f_1) &= -\bar{\D}(X_2\,f_2 + \dots +X_d\,f_d) \\
			&= 0,
	\end{align*}
	since $\o$ is closed.  But $\bar\D(X_1\,f_1) = X_1\,\bar\D(f_1) = 0$ implies that $\bar\D(f_1)=0$, since $\HH_1$ is not the trivial representation.
	
	Now,
	\begin{align*}
		\bar\D(f_1) &= \bar\D \left(\sum_{k \in J_{\nu_1}}\sum_{{\bl} \in J_{\bnu}} f_{1}(k,{\bl})\,u_{k}\otimes v_{\bl}\right) \\
			&= \sum_{k \in J_{\nu_1}}\left(\sum_{{\bl} \in J_{\bnu}} f_{1}(k,{\bl})\,\D(v_{\bl})\right)\,u_{k} \\
			&= 0
	\end{align*}
	implies that for each fixed $k \in J_{\nu_1}$,
	\begin{align*}
		\sum_{{\bl} \in J_{\bnu}} f_{1}(k,{\bl})\,\D(v_{\bl}) &= 0.
	\end{align*}
	Of course, then
	\[
		\D(f_{1}\mid_{k}) = \sum_{{\bl} \in J_{\bnu}} f_{1}(k,{\bl})\,\left\| u_{k} \right\|\,\D(v_{\bl}) = 0,
	\]
	proving the proposition.
\end{proof}

%==========================
%====================================================

We are now prepared to state the proof of Theorem \ref{b} for irreducible unitary representations.

%===============================================

\subsection{Irreducible case}

\begin{theorem} \label{irredb}
Let $\HH_1 \otimes\dots\otimes \HH_d$ be the Hilbert space of an irreducible representation of $\SL(2,\RR)^d$ with no trivial factor.  Suppose $s > 1$ and  $1\leq n\leq d-1$.  Then, for $0 < \mu_0 < \min\left\{\mu_{1},\dots,\mu_d\right\}$ and any $t < s-1$, there is a constant $C_{\mu_0,s,t}>0$ such that for any $n$-cocycle $\o \in \O_{\RR^d}^{n} (W^{s}(\HH_1\otimes\dots\otimes\HH_d))$ there exists $\eta \in \O_{\RR^d}^{n-1} (W^{t}(\HH_1\otimes\dots\otimes\HH_d))$ with $\di\eta = \o$ and
\begin{align*}
	\left\|\eta(X_{i_1},\dots,X_{i_{n-1}})\right\|_{t} &\leq C_{\mu_0,s,t}\min\left\{\left\| \o(X_{j_1},\dots,X_{j_n}) \right\|_{s_d} \right\}
\end{align*}
for all multi-indices $1\leq i_1 < \dots < i_{n-1} \leq d$, where the minimum is taken over all multi-indices $1\leq j_1 < \dots < j_{n} \leq d$ that become $i_1, \dots, i_{n-1}$ after omission of one index.
\end{theorem}

\begin{remark*}
As is the case with Theorem~\ref{irreda}, the bound on Sobolev norm in Theorem~\ref{irredb} only gives non-trivial information when $\o$ has Sobolev order at least $s_d$.  
\end{remark*}

\begin{proof}
By way of induction, suppose intermediate cohomology groups vanish for $\RR^p$-actions on $\HH_1\otimes\dots\otimes\HH_p$, where $2\leq p\leq d-1$.  The base case is the first cohomology over the $\RR^2$-action on $\HH_1 \otimes \HH_2$, known to vanish from \cite{M1}.

Let $\o_1$ be obtained as above.  Then for every $k\in J_{\nu_1}$, $(\o_{1}\mid_{k})$ is an $n$-form over the $\RR^{d-1}$-action on $\HH_2 \otimes\dots\otimes\HH_d$ generated by $X_2,\dots,X_d$.  By Lemma \ref{closed}, $(\o_{1}\mid_{k})$ is closed.

If $n < d-1$, the induction hypothesis implies that there exists an $(n-1)$-form $\eta_{1,k} \in \O_{\RR^{d-1}}^{n}(W^{t}(\HH_2 \otimes\dots\otimes\HH_d))$ with $\di\eta_{1,k} = (\o_{1}\mid_{k})$ and 
\[
\left\| \eta_{1,k}(X_{i_{1}},\dots,X_{i_{n-1}})\right\|_{t} \leq C_{\mu_{0},s,t}\,\min \left\{\left\| (\o_{1}\mid_{k})(X_{j_{1}},\dots,X_{j_{n}})\right\|_{s_{d-1}}\right\}
\]  
where the minimum is taken over multi-indices $2\leq j_{1} < \dots < j_{n}\leq d$ that become $i_1, \dots, i_{n-1}$ after omission of one index.

On the other hand, if $n = d-1$, then $(\o_{1}\mid_{k})$ is a top-degree form over this $\RR^{d-1}$-action, and Proposition \ref{prop} implies that 
\[
	(\o_{1}\mid_{k})(X_2,\dots,X_d) \in \ker\II_{X_2,\dots,X_d}^{s}
\]
which in turn implies, by Theorem \ref{irreda}, that there is an $(n-1)$-form $\eta_{1,k}$ with $\di\eta_{1,k} = (\o_{1}\mid_{k})$, and satisfying 
\begin{align*}
	\left\| \eta_{1,k}(X_{2},\dots,\widehat{X_{j}}, \dots,X_{d})\right\|_{t} &\leq C_{\mu_{0},s,t}\,\left\| (\o_{1}\mid_{k})(X_{2},\dots,X_{d})\right\|_{s_{d-1}}
\end{align*}  
for all $j = 2, \dots, d$.

Defining $\eta_1 \in \O_{\RR^d}^{n-1}(W^{t}(\HH_1 \otimes\dots\otimes\HH_{d}))$ so that $(\eta_{1}\mid_{k}) = \eta_{1,k}$ for all $k \in J_{\nu_1}$, we see that
\begin{align*}
	\di\eta_{1}(X_{i_1},\dots,X_{i_{n}}) &= \o_{1}(X_{i_1},\dots,X_{i_{n}})
\end{align*}
for $i_1 \geq 2$.  Also,
\begin{align*}
	&\left\| \eta_{1}(X_{i_1},\dots,X_{i_{n-1}}) \right\|_{t}^{2} \\
	&= \sum_{k \in J_{\nu_1}}\sum_{{\bl}\in J_{\bnu}} (1 + \mu_1 + \left\|\bmu\right\| + 2(k+ \frac{\e}{2})^2 + 2|{\bl}+ \frac{\bep}{2}|^2)^{t}\,|\eta_{1}(X_{i_1},\dots,X_{i_{n-1}})(k,{\bl})|^2\,\left\| u_k \otimes v_{\bl} \right\|^{2} \\
		&\leq \sum_{k \in J_{\nu_1}}\sum_{{\bl}\in J_{\bnu}} (1 + \mu_1 + 2(k+ \frac{\e}{2})^2)^{t}(1 + \left\|\bmu\right\| + 2|{\bl}+ \frac{\bep}{2}|^2)^{t}\,|\eta_{1}(X_{i_1},\dots,X_{i_{n-1}})(k,{\bl})|^2\,\left\| u_k \right\|^{2}\,\left\| v_{\bl} \right\|^{2} \\
		&= \sum_{k \in J_{\nu_1}}(1 + \mu_1 + 2(k+ \frac{\e}{2})^2)^{t} \left\| (\eta_{1}\mid_{k})(X_{i_1},\dots,X_{i_{n-1}}) \right\|_{t}^{2} \\
		&\leq C_{\mu_0,s,t}^{2}\sum_{k \in J_{\nu_1}}(1 + \mu_1 + 2(k+ \frac{\e}{2})^2)^{t} \min \left\{\left\| (\o_{1}\mid_{k})(X_{j_{1}},\dots,X_{j_{n}})\right\|_{s_{d-1}}^{2}\right\} \\
		&\leq C_{\mu_0,s,t}^{2}\min \left\{\left\| \o_{1}(X_{j_{1}},\dots,X_{j_{n}})\right\|_{s_d}^{2}\right\}.
\end{align*}

Repeating the above procedure, define $\eta_m \in \O_{\RR^d}^{n-1}(W^{t}(\HH_1 \otimes\dots\otimes\HH_{d}))$ such that 
\[
	\di\eta_{m}(X_{i_1},\dots,X_{i_{n}}) = \o_{m}(X_{i_1},\dots,X_{i_{n}})
\] 
for $m=1,\dots,d$, and the index $m \in \left\{1,\dots,d \right\}$ missing.  Exactly as above,
\begin{align*}
	\left\| \eta_{m}(X_{i_1},\dots,X_{i_{n-1}}) \right\|_{t}^{2} &\leq C_{\mu_0,s,t}^{2}\min \left\{\left\| \o_{m}(X_{j_{1}},\dots,X_{j_{n}})\right\|_{s_d}^{2}\right\},
\end{align*}
where the minimum is taken over all multi-indices $1\leq j_{1} < \dots < j_{n}\leq d$ (none equal to $m$) that become $i_1, \dots, i_{n-1}$ after omission of one index.

Setting 
\[
\eta := \frac{1}{d-n+1}(\eta_1 + \dots + \eta_d),
\]
 we have $\di\eta = \o$ and 
\begin{align*}
	\left\|\eta(X_{i_1},\dots,X_{i_{n-1}})\right\|_{t} &\leq C_{\mu_0,s,t}\min\left\{\left\| \o(X_{j_1},\dots,X_{j_n}) \right\|_{s_d}\right\}
\end{align*}
for all multi-indices $1\leq i_1 < \dots < i_{n-1} \leq d$, where the minimum is taken over all multi-indices $1\leq j_1 < \dots < j_{n} \leq d$ that become $i_1, \dots, i_{n-1}$ after omission of one index.  This proves the theorem.
\end{proof}

%===========================================%===========================================

%===============================================

\subsection{Proof of Theorem \ref{b}}

\begin{proof}
Let $\HH, \mu_0, s, t$ be as in the theorem statement, and $1 \leq n \leq d-1$.  Let 
\[
	\o \in \O_{\RR^d}^{n}(W^{s_d}(\HH))\quad\textrm{with}\quad \di\o = 0.
\]
There is a direct decomposition
\[
	W^{s_d}(\HH) = \int_{\oplus} W^{s_d}(\HH_\l)\,ds(\l)
\]
where $ds$-almost every $\HH_\l$ is irreducible and without trivial factors (by assumption), and $\o$ decomposes
\[
	\o(X_{i_1}, \dots, X_{i_n}) = \int_{\oplus} \o_{\l}(X_{i_1}, \dots, X_{i_n})\,ds(\l)
\]
such that $ds$-almost every $\o_\l$ is a cocycle in $\O_{\RR^d}^{n}(W^{s_d}(\HH_{\l}))$, where $\HH_{\l} = \HH_{\bmu}$ is irreducible.

For these $\l$, Theorem \ref{irredb} supplies $\eta_\l:=\eta_{\bmu} \in \O_{\RR^d}^{n-1}(W^{t}(\HH_{\bmu}))$ with $\di\eta_{\bmu} = \o_{\bmu}$ and
\[  
	\left\|\eta_{\bmu}(X_{i_1},\dots,X_{i_{n-1}})\right\|_{t} \leq C_{\mu_0,s,t}\, \min\{\left\| \o_{\bmu}(X_{j_1},\dots,X_{j_n}) \right\|_{s_d}\}
\]
where the minimum is taken over all $1 \leq j_1 <\dots<j_n\leq d$ that become $i_1, \dots, i_{n-1}$ after omission of one index.  For $\l$ where $\HH_{\l}$ is not irreducible without trivial factors, set $\eta_\l = 0$.  Defining
\[
	\eta(X_{i_1}, \dots, X_{i_{n-1}}) := \int_{\oplus} \eta_{\l}(X_{i_1}, \dots, X_{i_{n-1}})\,ds(\l),
\]
we have
\begin{align*}
	\left\|\eta(X_{i_1},\dots,X_{i_{n-1}})\right\|_{t}^{2} &= \int_{\oplus} \left\|\eta_{\l}(X_{i_1},\dots,X_{i_{n-1}})\right\|_{t}^{2}\, ds(\l) \\
		&\leq C_{\mu_0,s,t}^{2} \int_{\oplus} \min\{\left\| \o_{\bmu}(X_{j_1},\dots,X_{j_n}) \right\|_{s_d}^{2}\} \\
		&\leq C_{\mu_0,s,t}^{2}\, \min\{\left\| \o(X_{j_1},\dots,X_{j_n}) \right\|_{s_d}^{2}\}.
\end{align*}
Since all operators $X_1,\dots, X_d$ are decomposable, we have that $\eta \in \O_{\RR^d}^{n-1}(W^{t}(\HH))$ satisfies $\di\eta = \o$, proving the theorem.
\end{proof}
 
%==========================================

\section{Proofs of Theorems \ref{one} and \ref{two}}\label{sectionsix}

In this section we apply Theorems \ref{a} and \ref{b} to prove Theorems \ref{one} and \ref{two} from the Introduction (Section \ref{results}).  In fact, we prove versions of Theorems \ref{one} and \ref{two} for Sobolev spaces.  

For both of these theorems, we will need the left-regular representation of $\SL(2,\RR)^d$ on $\Lii(\SL(2,\RR)^d/\G)$ to satisfy the spectral gap assumption.  This is provided by the following theorem, which was proved for non-cocompact $\G$ by D. Kleinbock and G. Margulis in \cite{KM}, and for cocompact $\G$ by L. Clozel in \cite{Clo03}.

\begin{theorem}\label{spectralgapth}
	Let $G = G_1 \times \dots \times G_k$ be a product of noncompact simple Lie groups, and $\G \subset G$ an irreducible lattice.  Then the restriction of $\Lii(G/\G)$ to every $G_i$ has a spectral gap.  
\end{theorem}

In particular, this implies that if $\G \subset \SL(2,\RR)^d$ is an irreducible lattice, then the regular representation on $\Lii(\SL(2,\RR)^d/\G)$ has a spectral gap for each $\Box_i$.  

\begin{theorem}[Sobolev spaces version of Theorem \ref{one}] \label{sobone}
	Let $\G \subset \SL(2,\RR)^d$ be an irreducible lattice.  For $s > 1$ and $t < s-1$ there is a constant $C_{s,t}$ such that the following holds.  For any 
	\[
	f \in \II_{X_1,\dots,X_d}^{s_d}(\Lii(\SL(2,\RR)^{d}/\G)), 
	\]
	there exist functions 
	\[
	g_1,\dots,g_d \in W^{t}(\Lii(\SL(2,\RR)^{d}/\G))
	\]
	satisfying
	\begin{align*}
		f &= X_1\,g_1 + \dots + X_d\,g_d
	\end{align*}
	and the Sobolev estimates
	\[
		\left\| g_i \right\|_{t} \leq C_{s,t}\, \left\| f \right\|_{s_d}
	\]
	for $i = 1,\dots, d$.
\end{theorem}

\begin{proof}
	Since the left-regular representation of $\SL(2,\RR)$ on $\Lii(\SL(2,\RR)^d/\G)$ has spectral gap for the Casimir operator coming from each copy of $\SL(2,\RR)$, we can apply Theorem \ref{a}, setting $\HH = \Lii(\SL(2,\RR)^d/\G)$.
\end{proof}

\begin{proof}[Proof of Theorem \ref{one}]
	Theorem \ref{one} follows immediately, by noting that for any unitary representation on $\HH$, the space $\Cinf(\HH)$ of smooth vectors coincides with the intersection of all Sobolev spaces $W^{s}(\HH)$ of positive order $s \geq 0$.
\end{proof}

Now, for $\G \subset \SL(2,\RR)^d$ an irreducible lattice, let  $\Lii_{0}(\SL(2,\RR)^d/\G)$ be the orthogonal complement to the constant functions in $\Lii(\SL(2,\RR)^d/\G)$.  Similarly, let $W_{0}^{s}(\SL(2,\RR)^d/\G)$ be the orthogonal complement to the constant functions in $W^{s}(\SL(2,\RR)^d/\G)$.  The following is a version of Theorem \ref{two} for forms taking values in Sobolev spaces.

\begin{theorem}[Sobolev spaces version of Theorem \ref{two}] \label{sobtwo}
Let $\G \subset \SL(2,\RR)^d$ be an irreducible lattice.  For any $s > 1$ and $t < s-1$, there is a constant $C_{s,t}$ such that the following holds.  For $1\leq n\leq d-1$ and any $n$-cocycle $\o \in \O_{\RR^d}^{n} (W_{0}^{s_d}(\SL(2,\RR)^d/\G))$, there exists $\eta \in \O_{\RR^d}^{n-1} (W_{0}^{t}(\SL(2,\RR)^d/\G))$ with $\di\eta = \o$ and
\begin{align*}
	\left\|\eta(X_{i_1},\dots,X_{i_{n-1}})\right\|_{t} &\leq C_{\mu_0,s,t}\min\left\{\left\| \o(X_{j_1},\dots,X_{j_n}) \right\|_{s_d}\right\}
\end{align*}
for all multi-indices $1\leq i_1 < \dots < i_{n-1} \leq d$, where the minimum is taken over all multi-indices $1\leq j_1 < \dots < j_{n} \leq d$ that become $i_1, \dots, i_{n-1}$ after omission of one index.
\end{theorem}

\begin{proof}
	The left-regular representation of $\SL(2,\RR)^d$ on $\Lii_{0}(\SL(2,\RR)^d/\G)$ has spectral gap for the Casimir operator coming from each of the factors of $\SL(2,\RR)^d$.  Furthermore, since the flows of $X_1,\dots, X_d$ are ergodic on $\SL(2,\RR)^d/\G$, we know that the direct decomposition
	\[
		\Lii_{0}(\SL(2,\RR)^d/\G) = \int_{\oplus} \HH_{\l}\,ds(\l)
	\]
	is such that $ds$-almost every $\HH_\l$ is irreducible without trivial factors.  We now apply Theorem \ref{b}.  
\end{proof}

\begin{proof}[Proof of Theorem \ref{two}]
	Theorem \ref{sobtwo} immediately implies that any smooth $n$-cocycle 
	\[
		\o \in \O_{\RR^d}^{n}(\Cinf(\Lii(\SL(2,\RR)^d/\G)))
	\]
	is cohomologous to the constant form $\o_c$ given by
	\[
		\o_{c}(X_{i_1},\dots,X_{i_n}) = \int_{\SL(2,\RR)^d/\G} \o(X_{i_1},\dots,X_{i_n})\,dm_{Haar}.
	\]
	This proves Theorem \ref{two}.
\end{proof}

\begin{remark*}
	If the lattice $\G$ is cocompact, then in Theorem \ref{two} we can replace the space of smooth vectors for the representation $\Cinf(\Lii(\SL(2,\RR)^d/\G))$ with the space of smooth functions $\Cinf(\SL(2,\RR)^d/\G)$.
\end{remark*}

%==========================================

\appendix

\section{On the passage from $\PSL(2,\RR)$ to $\SL(2,\RR)$ and proof of Lemma~\ref{degreeoneblue}} \label{appendix}

For the sake of completeness, we elaborate on the proof of Theorem \ref{m}.  Since it was originally proved for $\PSL(2,\RR)$, we include here the observations necessary for the same proof to apply to $\SL(2,\RR)$.

%=============================================

\subsection{Difference equation}

Let $\HH_\mu$ be the Hilbert space of an irreducible unitary representation of $\SL(2,\RR)$ with Casimir parameter $\mu$.  For a given $f \in \II_{X}^{s}(\HH_\mu)$, where $s>1$, we would like to solve the coboundary equation
\begin{align}
	X\,g &= f.
\end{align}
Expressing $f$ in the basis defined in Section \ref{orthobasis}, and applying Lemma \ref{action}, the coboundary equation becomes the difference equation
\begin{align}
	b^{+}(k-1)g(k-1) - b^{-}(k+1)g(k+1) &= f(k), \label{dE}
\end{align}
where the sequences $b^{+}(k)$ and $b^{-}(k)$ are those defined in Section \ref{actionsection}.

If $\HH_\mu$ is from the discrete series, we write the difference equation as
\begin{align}\label{dEdiscrete}
	\begin{cases} f(n+k+1) = b^{+}(n+k)g(n+k) - b^{-}(n+k+2)g(n+k+2), &\textrm{for all } k \geq 0; \\
	f(n) = -b^{-}(n+1)g(n+1). & \end{cases}
\end{align}

%=============================================

\subsection{Calculations in the principal and complementary series}

Equation \eqref{dE} is exactly the difference equation that is solved in \cite{M2}.  There, Mieczkowski finds solutions $g^{0}$ and $g^{1}$ to the homogeneous equation $X\,g=0$ in the principal and complementary series, defined by 
\begin{equation}\label{g0formula}
	g^{0}(2k) = \prod_{j=0}^{|k|-1}\frac{b^{+}(2j)}{b^{-}(2j+2)}, \quad g^{0}(2k+1) = 0
\end{equation}
and
\begin{equation}\label{g1formula}
	g^{1}(2|k|+1) = g^{1}(-2|k|-1) = \prod_{j=1}^{|k|}\frac{b^{+}(2j-1)}{b^{-}(2j+1)}, \quad g^{1}(2k) = 0,
\end{equation}
with initial values $g^{0}(0)=1$, $g^{0}(1)=0$, $g^{1}(0)=0$, and $g^{1}(1)=1$.  It is easy to check that these satisfy the homogeneous equation.

The next step is to construct the Green's function
\[
	G(k,l) = \frac{\det\begin{pmatrix} g^{0}(l) & g^{1}(l) \\ g^{0}(k) & g^{1}(k) \end{pmatrix}}{\det\begin{pmatrix} g^{0}(l) & g^{1}(l) \\ g^{0}(l+1) & g^{1}(l+1) \end{pmatrix}}
\]
from which one can then write a formal solution to \eqref{dE} as
\begin{align*}
	g(2k) =  - \sum_{l\leq k} \frac{G(2k, 2l-1)}{b^{-}(2l)}\, f(2l-1), \quad g(2k+1) =  - \sum_{l\leq k} \frac{G(2k+1, 2l)}{b^{-}(2l+1)}\, f(2l).
\end{align*}

It is only left to check that this formal solution satisfies the Sobolev estimates from the theorem.  This is done by studying the asymptotic behavior of the homogeneous solutions $g^{0}$ and $g^{1}$.  Mieczkowski \cite{M2} has estimates on the asymptotic properties of these homogeneous solutions for representations of $\PSL(2,\RR)$, \emph{i.e.} for the case where $\e =0$.  We need only show that the same estimates apply to the case $\e = 1$ (in the second principal and second discrete series).  For the second principal series, this follows from the following lemma, comparing the two cases.

\begin{lemma} \label{comparison}
	For $\nu \in i\RR$, and for all $k > 0$, 
	\[
		|g_{\e=0}^{0}(2k)| \leq |g_{\e=1}^{0}(2k)| \leq |g_{\e=0}^{1}(2k+1)| \leq |g_{\e=1}^{1}(2k+1)| \leq |g_{\e=0}^{0}(2k+2)|.
	\]
\end{lemma}

\begin{proof}
	This is routine calculation.
\end{proof}

From \cite{M2} we have the following lemmas.

\begin{lemma}[4.3 from \cite{M2}] \label{mlemmaprincipal}
	For the principal series, $\nu \in i\RR$, we have,
	\begin{align*}
		C_{\nu}^{-1}((4|k|+1)^2 + |\nu|^{2})^{-\frac{1}{2}} \leq |g^{0}(2k)|^2 \leq C_{\nu}((4|k|+3)^2 + |\nu|^{2})^{-\frac{1}{2}}
	\end{align*}
	and
	\begin{align*}
		C_{\nu}^{-1}((4|k|-1)^2 + |\nu|^{2})^{-\frac{1}{2}} \leq |g^{1}(2k+1)|^2 \leq C_{\nu}((4|k|+5)^2 + |\nu|^{2})^{-\frac{1}{2}}
	\end{align*}
	where $C_{\nu}$ is bounded in $\nu$.
\end{lemma}

\begin{lemma}[4.5 from \cite{M2}] \label{mlemmacomplementary}
	For the complementary series, $-1 < \nu < 1$, $\nu \neq 0$, we have,
	\begin{align}
		\frac{1+\nu}{3-\nu} \cdot \left(\frac{4|k|-3+\nu}{1+\nu}\right)^{\frac{\nu-1}{2}} \leq |g^{0}(2k)| \leq \left(\frac{4|k|+3-\nu}{3-\nu}\right)^{\frac{\nu-1}{2}}
	\end{align}
	and
	\begin{align}
		\frac{3+\nu}{5-\nu} \cdot \left(\frac{4|k|-1+\nu}{3+\nu}\right)^{\frac{\nu-1}{2}} \leq |g^{1}(2k+1)| \leq \left(\frac{4|k|+5-\nu}{5-\nu}\right)^{\frac{\nu-1}{2}}.
	\end{align}
\end{lemma}

After combining Lemma \ref{comparison} with the Lemma \ref{mlemmaprincipal}, one can carry out Mieczkowski's proof to show that the desired Sobolev estimates hold for the formal solutions to the coboundary equation in representations from the second principal series.

%==============================================================

\subsection{Calculations in the discrete series}

In a representation on $\HH_\mu$ from the discrete series, we have $\mu \in \{-n^2+n\}\cup\{-n^2+\frac{1}{4}\}$.  Define
\begin{equation}\label{g0disc}
	g^{0}(n+2k) = \prod_{j=0}^{k-1}\frac{b^{+}(n+2j)}{b^{-}(n+2j+2)}, \quad g^{0}(n+2k+1) = 0
\end{equation}
and
\begin{equation}\label{g1disc}
	g^{1}(n+2k+1) = \prod_{j=1}^{k}\frac{b^{+}(n+2j-1)}{b^{-}(n+2j+1)}, \quad g^{1}(n+2k) = 0,
\end{equation}
with initial values $g^{0}(n)=1$, $g^{0}(n+1)=0$, $g^{1}(n)=0$, $g^{1}(n+1)=1$.  Note that $g^1$ no longer solves the homogeneous version of Equation \eqref{dEdiscrete}, because $(X\,g^1)(n) = -b^{-}(n+1) = 1/2 \neq 0$.  Still it is useful to define $g^1$ as above, and to estimate its asymptotic behavior.

Our proof of the following lemma includes the second holomorphic discrete series.

\begin{lemma}[4.6 from \cite{M2}] \label{mlemmadiscrete}
	For the discrete series, $\nu = 2n+\e-1$, we have that 
	\begin{align*}
		\left(\frac{2k+\nu+1}{\nu+1} \right)^{\frac{\nu-1}{2}} \leq |g^{0}(n+2k)| \leq \left(\frac{\nu+1}{2}\right)\, k^{\frac{\nu-1}{2}}
	\end{align*}
	and
	\begin{align*}
		\left(\frac{2k+\nu-1}{\nu-1} \right)^{\frac{\nu-1}{2}} \leq |g^{1}(n+2k+1)| \leq \frac{\nu+1}{3}\,(2k-1)^{\frac{\nu-1}{2}}.
	\end{align*}	
\end{lemma}

\begin{proof}
	Combining $\nu = 2n + \e - 1$ with the definition of $g^0$, we have
	\begin{align*}
		g^{0}(n+2k) &= \prod_{j=0}^{k-1}\left(1 + \frac{\nu -1}{2(j+1)}\right).
	\end{align*}
	Taking logarithms, and employing the inequality
	\begin{align}
		\frac{x}{1+x} \leq \log(1+x) \leq x, \label{loginequality}
	\end{align}
	we have
	\[
		\frac{\nu-1}{2}\sum_{j=0}^{k-1}\frac{1}{(j+1)+\frac{\nu-1}{2}} = \sum_{j=0}^{k-1}\frac{\nu-1}{2(j+1)+\nu-1} \leq \log|g^{0}(2k)|.
	\]
	The left-hand side can be estimated by the integral inequality
	\begin{align*}
		\frac{\nu-1}{2}\sum_{j=0}^{k-1}\frac{1}{(j+1)+\frac{\nu-1}{2}} &\geq \frac{\nu-1}{2}\int_{0}^{k} \frac{dx}{(x+1)+\frac{\nu-1}{2}} \\
			&= \frac{\nu-1}{2}\log\left(\frac{2k + 1+\nu}{1+\nu}\right).
	\end{align*}
	Exponentiating gives the first lower bound.
	
	For the first upper bound, we write
	\begin{align*}
		g^{0}(n+2k) &= \left(\frac{\nu+1}{2}\right)\prod_{j=1}^{k-1}\left(1 + \frac{\nu-1}{2(j+1)}\right),
	\end{align*}
	and use the second part of \eqref{loginequality} to estimate the product:
	\begin{align*}
		\log\left|\prod_{j=1}^{k-1}\left(1 + \frac{\nu-1}{2(j+1)}\right)\right| &= \sum_{j=1}^{k-1} \log \left(1 + \frac{\nu-1}{2(j+1)}\right) \leq \sum_{j=1}^{k-1} \frac{\nu-1}{2(j+1)},
	\end{align*}
	and this we can bound with the integral
	\begin{align*}
		\sum_{j=1}^{k-1} \frac{\nu-1}{2(j+1)} &\leq \frac{\nu-1}{2}\int_{0}^{k-1}\frac{1}{x+1}\,dx = \frac{\nu-1}{2}\log k.
	\end{align*}
	Exponentiating yields the upper bound.
	
	The same procedure will give the estimates for $g^{1}(n+2k+1)$.
\end{proof}

%================================================

\subsubsection{Solution to coboundary equation in discrete series}

In a representation from the discrete series, where $\mu \in \{-n^2 + n\}\cup\{-n^2 + \frac{1}{4}\}$, there is only one $X$-invariant distribution up to multiplication by a constant: 
\[
	\D_{0}(u_{n+2k+1}) = 0, \quad\D_{0}(u_{n+2k}) = \prod_{i=1}^{k} \b(n+2i-1).
\]
As always, empty products are set to $1$, meaning that $\D_{0}(u_n)=1$.  The distribution $\D_{1}$, defined by
\[
	\D_{1}(u_{n+2k}) = 0, \quad \D_{1}(u_{n+2k -1}) = \prod_{i=1}^{k-1} \b(n+2i),
\]
is not $X$-invariant because $\D_{1}(X\,u_{n}) = \D_{1}((n+\frac{\e}{2})\,u_{n+1}) = (n+\frac{\e}{2}) \neq 0$.

Given $f \in \ker\D_{0} \cap W^{s}(\HH_\mu)$, we would like to solve the coboundary equation defined by the difference equation \eqref{dEdiscrete}.  Notice that we can express any $f \in W^{s}(\HH_\mu)$ as
\begin{align*}
	f &= \sum_{k = 0}^{\infty} f(n + 2k)\,u_{n+2k} + \sum_{k = 0}^{\infty} f(n + 2k + 1)\,u_{n+2k+1} \\
		&:= f_{even} + f_{odd},
\end{align*}
and that one always has $f_{odd} \in \ker\D_{0}$.  Therefore, $f \in \ker\D_{0}$ if and only if $f_{even} \in \ker\D_{0}$.

Our strategy is to solve the coboundary equation for $f_{even}$ and $f_{odd}$ separately.

\begin{proposition}[Solution for $f_{even}$] \label{feven}
Let $s>1$, $t<s-1$.  There is a constant $C_{s,t}$ such that for any $f \in \ker\D_{0}\cap W^{s}(\HH_\mu)$ of the form  
\begin{align*}
	f &= \sum_{k=0}^{\infty} f(n+2k)\,u_{n+2k},
\end{align*}
there exists $g \in W^{t}(\HH_{\mu})$ such that $X\,g = f$ and $\left\| g \right\|_{t} \leq C_{s,t}\cdot\left\| f \right\|_{s}$.
\end{proposition}

\begin{proof}
From the second part of the difference equation \eqref{dEdiscrete}, one sees that $g(n+1)$ is determined by $f(n)$.  One can use this to solve Equation \eqref{dEdiscrete} for $g(n+3)$ and successively for $g(n+2k+1)$, arriving at the formula
\begin{align*}
	-g(n+2k+1) &= \frac{1}{b^{-}(n+1)}\,g^{1}(n+2k+1)\,\sum_{i=0}^{k} f(n+2i)\,\D_{0}(u_{n+2i}).
\end{align*}    
Now, by our assumption that $D_{0}(f) = 0$,   
\begin{align} 
	g(n+2k+1) &= \frac{1}{b^{-}(n+1)}\,g^{1}(n+2k+1)\,\sum_{i=k+1}^{\infty} f(n+2i)\,\D_{0}(u_{n+2i}). \label{godd}
\end{align}    
We can set $g(n+2k)=0$ for all $k \geq 0$.

It is only left to compute the Sobolev norm of $g$.  At this point we refer to estimates in \cite{M2} that prove the lemma.
\end{proof}

One must be more careful in treating $f_{odd}$, because it is automatically in the kernel of all invariant distributions.  We cannot assume that $f_{odd} \in \ker\D_1$, because $\D_1$ is not $X$-invariant in representations from the discrete series.  We must therefore find another way to arrive at an expression similar to \eqref{godd}.

\begin{proposition}[Solution for $f_{odd}$] \label{fodd}
Let $s>1$, $t<s-1$.  There is a constant $C_{s,t}$ such that for any $f \in W^{s}(\HH_\mu)$ of the form  
\begin{align*}
	f &= \sum_{k=0}^{\infty} f(n+2k+1)\,u_{n+2k+1},
\end{align*}
there exists $g \in W^{t}(\HH_{\mu})$ such that $X\,g = f$ and $\left\| g \right\|_{t} \leq C_{s,t}\cdot\left\| f \right\|_{s}$.
\end{proposition}

\begin{proof}

A formal solution is given by
\begin{align*}
	&-g(n+2k) \\
		&= \frac{1}{b^{-}(n+2k)}\,\prod_{m=1}^{k-1}\frac{b^{+}(n+2m)}{b^{-}(n+2m)} \left[\sum_{i=1}^{k} f(n+2i-1)\,\prod_{m=1}^{i-1} \frac{b^{-}(n+2m)}{b^{+}(n+2m)} - b^{+}(n)g(n)\right],
\end{align*}
and we are left to choose $g(n)$ (unlike in Proposition~\ref{feven}, where the first value of the solution was dictated by the difference equation).  The natural choice is
\begin{align*}
	b^{+}(n)g(n) &= \sum_{i=1}^{\infty} f(n+2i-1)\,\prod_{m=1}^{i-1} \frac{b^{-}(n+2m)}{b^{+}(n+2m)},
\end{align*}
so that we will be able to apply the same estimates here as the ones that conclude the proof of Proposition~\ref{feven}.  We need to first see that this choice of $g(n)$ is finite:
\begin{align*}
	&|b^{+}(n)g(n)|^2 \\
		&= \left|\sum_{i=1}^{\infty} f(n+2i-1)\,\prod_{m=1}^{i-1} \b(n+2m)\right|^2 \\
		&\leq \sum_{i=1}^{\infty} (1+\mu + 2(n+2i-1 + \frac{\e}{2})^2)^{s} |f(n+2i-1)|^2\,\left\| u_{n+2i-1} \right\|^2\\
		&\indent\times\sum_{i=1}^{\infty}\,(1+\mu + 2(n+2i-1 + \frac{\e}{2})^2)^{-s}\left| b^{+}(n) \right|^2\,\left|\prod_{m=0}^{i-1}\frac{b^{-}(n+2m+2)}{b^{+}(n+2m)}\right|^2\,\left\| u_{n+2i-1} \right\|^{-2},
\intertext{and by Lemmas \ref{mlemmadiscrete} and \ref{fflemma},}
	&\leq \left\| f \right\|_{s}^{2}\cdot\sum_{i=1}^{\infty}(1+\mu + 2(n+2i-1 + \frac{\e}{2})^2)^{-s}\,\left(n+\frac{\e}{2}\right)^2\,\left(\frac{2i +\nu+1}{\nu+1} \right)^{-\nu+1} \\
		&\indent\times C\left(\frac{2i}{n+1}\right)^{\nu}.
\end{align*}
The sum converges because $s > 1$.  This shows that $|g(n)| < \infty$, as we have chosen it.

We now have 
\begin{align*}
	g(n+2k) &= \frac{1}{b^{+}(n)}\, g^{0}(n+2k)\, \sum_{i=k+1}^{\infty} f(n+2i-1)\,\D_{1}(u_{n+2i-1}).
\end{align*}
Again, estimates from \cite{M2} yield the desired bound on the Sobolev norms of $g$.
\end{proof}

Propositions~\ref{feven} and~\ref{fodd} imply Theorem \ref{m} for representations of $\SL(2,\RR)$ from the first and second holomorphic discrete series.  This completes the extension of Theorem \ref{m} from $\PSL(2,\RR)$ to $\SL(2,\RR)$.

%===============================================

\subsection{Proof of Lemma~\ref{degreeoneblue}}\label{lemmaproof}

In this section we prove Lemma~\ref{degreeoneblue}, which is needed in the proof of Lemma~\ref{sobolevnorms}.  We will use the following simple fact.

\begin{lemma}\label{simplefact}
There are positive constants $C_1, C_2, C_3, C_4, C_5, C_6$ such that in any representation from the complementary series,
\[
C_1\,(1+\mu+2(2k)^2)^{\frac{1}{2}} \leq 4k + 1 + \nu \leq C_2 \,(1+\mu+2(2k)^2)^{\frac{1}{2}}
\]
and
\[
C_3\,(1+\mu+2(2k+1)^2)^{\frac{1}{2}} \leq 4k + 1 - \nu \leq C_4\,(1+\mu+2(2k+1)^2)^{\frac{1}{2}}
\]
and in any representation from the discrete series,
\[
C_5\,(1+\mu+2(n+2k+\e/2)^2)^{\frac{1}{2}} \leq 2k + 1 + \nu \leq C_6\,(1+\mu+2(n+2k+\e/2)^2)^{\frac{1}{2}},
\]
for all $k \geq 0$.
\end{lemma}

\begin{remark*}
We omit the details of the proof.  One only needs to take the relationship $\nu^2 = 1 - 4\mu$ into account, as well as the different possible values $\nu$ and $\mu$ can take in the complementary and discrete series representations.  
\end{remark*}

\begin{proof}[Proof of Lemma~\ref{degreeoneblue}]
By comparing~\eqref{D0formula} to~\eqref{g1formula}, we see that in the first principal series representations, where $\e=0$ and $\nu \in i\RR$,
\[
	\frac{\abs{\D_0^{\HH_{\mu}}(u_{\pm2k})}^2}{\norm{u_{\pm 2k}}^2} = \Abs{g^1 (2k+1)}^{-2}\Abs{\frac{1-\nu}{4k+1-\nu}}^2
\]
where $k \in \ZZ_{\geq 0}$.  We use Lemmas~\ref{fflemma} and~\ref{mlemmaprincipal} to get upper bound
\[
\frac{\abs{\D_0^{\HH_{\mu}}(u_{\pm(2k)})}^2}{\norm{u_{\pm(2k)}}^2} \leq C_{\nu} \Abs{4k + \nu - 1}\Abs{\frac{1-\nu}{4k + 1-  \nu}}^2 \leq C_{\nu}\,\frac{1-\nu^2}{\left((4k +1)^2-\nu^2\right)^{1/2}}.
\]
where $C_{\nu}$ is the constant from Lemma~\ref{mlemmaprincipal}, bounded over $\nu \in i\RR$.  Since $1-\nu^2 = 4\mu$ we can easily bound this above by
\[
	\leq C_{\nu}\,\frac{\mu}{\left(1 + \mu + 2 (2k)\right)^{1/2}}
\]
after some appropriate scaling of $C_{\nu}>0$.  Getting the lower bound is entirely similar, as are the corresponding calculations for $\frac{\abs{\D_1^{\HH_{\mu}}(u_{\pm(2k+1)})}^2}{\norm{u_{\pm(2k+1)}}^2}$ (where we would compare~\eqref{D1formula} and~\eqref{g0formula} in the first step) and for the second principal series (where $\e=1$ and we use Lemma~\ref{comparison}).

For the complementary series, where $\nu \in (-1,1)\backslash\{0\}$, we first use~\eqref{D0formula} and~\eqref{Pidef} to write
\begin{align*}
\frac{\norm{u_{\pm 2k}}^2}{\abs{\D_0^{\HH_{\mu}}(u_{\pm 2k})}^2} &= \prod_{j=1}^{k}\Abs{\frac{b^+(2j-1)}{b^-(2j-1)}}^2 \, \prod_{j =1}^{2k}\Abs{\frac{b^- (j)}{b^+ (j-1)}} \\
	&= \Abs{\frac{b^+(2k)}{b^+(0)}}\,\prod_{j=1}^{k}\Abs{\frac{b^+(2j-1)}{b^-(2j-1)}} \, \prod_{j =1}^{k}\Abs{\frac{b^- (2j)}{b^+ (2j)}}
\end{align*}
for all $k \geq 0$.  After accounting for the definitions of $b^+$ and $b^-$, and rearranging terms, we have
\begin{equation}\label{combwith}
\frac{\norm{u_{\pm 2k}}^2}{\abs{\D_0^{\HH_{\mu}}(u_{\pm 2k})}^2} = \frac{4k + 1 + \nu}{1+\nu}\prod_{j=1}^{k}\left(1+\frac{4+4\nu}{(4j+1+\nu)(4j-3-\nu)}\right).
\end{equation}
Similarly, by using~\eqref{D1formula} and~\eqref{Pidef} one can show that 
\begin{equation}\label{combwith2}
\frac{\norm{u_{\pm2k\pm1}}^2}{\abs{\D_1^{\HH_{\mu}}(u_{\pm2k\pm1})}^2} = \frac{4k + 1 - \nu}{1+\nu}\prod_{j=1}^{k}\left(1+\frac{4+4\nu}{(4j+1+\nu)(4j-3-\nu)}\right)^{-1}.
\end{equation}
The inequality $\frac{x}{1+x}\leq\log(1+x)\leq x$ for $x>0$ lets us bound the logarithm of the product in~\eqref{combwith} below by
\[
\sum_{j=1}^{k}\frac{4+4\nu}{4+4\nu+(4j+1+\nu)(4j-3-\nu)} \leq\log\prod_{j=1}^{k}\left(1+\frac{4+4\nu}{(4j+1+\nu)(4j-3-\nu)}\right)
\]
and above by
\[
\log\prod_{j=1}^{k}\left(1+\frac{4+4\nu}{(4j+1+\nu)(4j-3-\nu)}\right) \leq \sum_{j=1}^{k}\frac{4+4\nu}{(4j+1+\nu)(4j-3-\nu)}.
\]
Comparing the sums to integrals and exponentiating, we see that there is a number $C_{\nu_0}>0$ such that 
\[
C_{\nu_0}^{-1}\leq\prod_{j=1}^{k}\left(1+\frac{4+4\nu}{(4j+1+\nu)(4j-3-\nu)}\right) \leq C_{\nu_0}
\]
whenever $\abs{\nu} \leq \nu_0$, where $\nu_0 < 1$.  Combining with~\eqref{combwith} and~\eqref{combwith2} proves the desired inequality after observing Lemma~\ref{simplefact}.

We now address the discrete series, where $\nu = 2n+\e -1$.  Combining~\eqref{D0formula} with~\eqref{Pidef}, 
\begin{align}
\frac{\norm{u_{n+2k}}^2}{\abs{\D_0^{\HH_{\mu}}(u_{n+2k})}^2} &= \frac{b^+(n+2k)}{b^+ (n)}\,\prod_{j=1}^{k}\frac{2j}{2j-1}\,\prod_{j=1}^{k}\frac{2j+\nu}{2j+1+\nu} \nonumber \\
	&= \frac{2k+1+\nu}{1+\nu}\,\prod_{j=1}^{k}\left(1 + \frac{1+\nu}{(2j-1)(2j+1+\nu)}\right) \label{togetherw}
\end{align}
The inequality $\frac{x}{1+x}\leq\log(1+x)\leq x$ for $x>0$ implies that the logarithm of the product is bounded below by
\[
\sum_{j=1}^{k} \frac{1+\nu}{1+\nu+(2j-1)(2j+1+\nu)}\leq\log\prod_{j=1}^{k}\left(1 + \frac{1+\nu}{(2j-1)(2j+1+\nu)}\right)
\]
and above by
\[
\log\prod_{j=1}^{k}\left(1 + \frac{1+\nu}{(2j-1)(2j+1+\nu)}\right) \leq \sum_{j=1}^{k} \frac{1+\nu}{(2j-1)(2j+1+\nu)}.
\]
Comparing the sums to integrals and exponentiating, we conclude that there is a number $C>0$ such that the inequality
\[
C^{-1}\,(\nu+1)^{\frac{1}{2}} \leq \prod_{j=1}^{k}\left(1 + \frac{1+\nu}{(2j-1)(2j+1+\nu)}\right) \leq C\,(\nu+1)^{\frac{1}{2}}
\]
is satisfied.  This and~\eqref{togetherw} imply that 
\[
	C^{-1}\,\frac{(\nu+1)^{\frac{1}{2}}}{2k+1+\nu}\leq\frac{\norm{u_{n+2k}}^2}{\abs{\D_0^{\HH_{\mu}}(u_{n+2k})}^2}\leq C\,\frac{(\nu+1)^{\frac{1}{2}}}{2k+1+\nu},
\]
which in turn implies the lemma.
\end{proof}

%===============================================

\subsection*{Acknowledgments}

The author is grateful to Livio Flaminio, Alexander Gorodnik, Anatole Katok, Svetlana Katok, and Ralf Spatzier for fruitful discussions and advice during the development of this article, and to the referee for valuable comments and suggestions.

%=============================================

\bibliographystyle{amsalpha}
\bibliography{../bibliography}
\end{document}